\documentclass[12pt,a4paper]{amsart}
\usepackage{cite}
\usepackage{comment}
\usepackage[T1]{fontenc}
\usepackage{lmodern}
\usepackage{amsmath}
\usepackage{amsfonts}
\usepackage{amssymb}
\usepackage{mathtools}
\usepackage{latexsym}
\usepackage{parskip}
\usepackage{enumerate}
\usepackage{lipsum}
\usepackage[utf8]{inputenc}
\usepackage{amsmath}
\usepackage{amsfonts}
\usepackage{amssymb}
\usepackage{latexsym}

\usepackage{verbatim}

\usepackage{multicol}

\usepackage[normalem]{ulem}

\usepackage{parskip}
\usepackage{fullpage}
\usepackage[pdftex]{graphicx}

\usepackage{enumerate}
\numberwithin{equation}{section}

\usepackage{amsthm}

\DeclareMathOperator{\e}{e}

\newtheorem{theorem}{Theorem}[section]
\newtheorem{corollary}[theorem]{Corollary}

\newtheorem{lemma}[theorem]{Lemma}

\theoremstyle{definition}
\newtheorem{definition}[theorem]{Definition}
\newtheorem{remark}[theorem]{Remark}
\newtheorem{example}[theorem]{Example}

\newcommand{\ep}{\varepsilon}

\newcommand{\N}{\mathbb{N}}
\newcommand{\R}{\mathbb{R}}
\newcommand{\Id}{\mathrm{Id}}

\newcommand{\Ker}{\mathrm{Ker}}

\newcounter{exemplo}[section]

\newcommand{\ovtheta}{\overline{\theta}}

\usepackage[usenames,dvipsnames,svgnames,table]{xcolor}

\begin{document}
\title{Nonuniform dichotomies via time rescaling}
\author{Davor Dragi\v{c}evi\'c}
\address{Faculty of Mathematics, University of Rijeka, Croatia}
\email[Davor Dragi\v cevi\' c]{ddragicevic@math.uniri.hr}
\author{C\'{e}sar M. Silva}
\address{Centro de Matem\'atica e Aplica\c{c}\~oes\\
Universidade da Beira Interior\\
6201-001 Covilh\~a\\
Portugal}
\email[C\'{e}sar M. Silva]{csilva@ubi.pt}


\maketitle

\begin{abstract}
We establish characterizations of general nonuniform dichotomies for linear nonautonomous dynamical systems in terms of nonuniform exponential dichotomies of appropriately time-rescaled systems. In discrete time, for a class of slowly growing rates, such a characterization requires an additional condition involving a suitable notion of nonuniform ordinary dichotomy, which can be omitted for invertible systems. In continuous time, we obtain a more general characterization for a large class of growth rates, without imposing this additional condition. Our approach does not rely on adapted Lyapunov norms and applies to evolution families on arbitrary Banach spaces. As applications, we show how several results concerning general nonuniform dichotomies can be derived from their counterparts for exponential dichotomies. These include results on robustness, admissibility, Lyapunov functions, smooth linearization, stable manifolds, and Sacker–Sell spectral theory.
\end{abstract}
\section{Introduction}

The notion of exponential dichotomy, introduced by Perron~\cite{Per}, plays an important role in the qualitative theory of nonautonomous dynamical systems. As with the classical notion of hyperbolicity for autonomous dynamical
systems, an exponential dichotomy requires, at each time, a splitting of the phase space
into two complementary subspaces, called the stable and unstable subspaces. The dynamics
contracts exponentially forward in time along the stable subspace and expands exponentially
forward in time along the unstable subspace. However, due to the flexibility of the nonautonomous setting, it is natural to consider more general asymptotic behavior in which the rates of contraction and expansion are not necessarily exponential. To the best of our knowledge, these generalized dichotomies were first studied by Martin~\cite{Martin}, Muldowney~\cite{M} and 
 Naulin and Pinto~\cite{NP}. More recently, Barreira and Valls~\cite{BV-DCDS} initiated the study of \emph{nonuniform} dichotomies with general growth rates, with subsequent works of Bento and Silva~\cite{BS2, BS4} considered even more general classes of nonuniform dichotomies. We refer to:
 \begin{itemize}
\item \cite{BVJFA,BS3,BS4} for results dealing with  invariant manifold theory of semilinear nonautonomous systems under the assumption that the linear part  admits a general nonuniform dichotomy;
\item \cite{BD, D1, D2, DLP,SilvaCPAA} for characterizations of general dichotomies in terms of the so-called admissibility property;
\item \cite{B,BDV2, CJ, JG,Silva} for  development of the Sacker-Sell spectral theory (and its applications) with respect to general dichotomies;
\item \cite{BDZ, BV-09, CR, DS, Torres} for linearization results devoted to nonautonomous nonlinear systems whose linear part admits a general nonuniform dichotomy;
\item \cite{BD,BV,Chu,SilvaCPAA} for results dealing with the robustness property of general nonuniform dichotomies;
\item comparison of generalized dichotomies associated to different growth rates~\cite{GJ,JLR}.
 \end{itemize}

 In our previous work~\cite{DS} (building on the earlier work~\cite{DSS1} on polynomial dichotomies), we showed that, for discrete nonautonomous dynamics and a class of growth rates $\mu$ exhibiting slow growth, a uniform $\mu$-dichotomy can be characterized in terms of ordinary dichotomy and an exponential dichotomy of another discrete system obtained from the original one via appropriate time rescaling. Similar results for the case of continuous time have been obtained in~\cite{DD}.

 However, the results from~\cite{DS} are only partially applicable to the case of nonuniform general dichotomies. More precisely, if a system exhibits a nonuniform dichotomy with a slow growth rate $\mu$, then the associated system constructed using time rescaling exhibits nonuniform exponential dichotomy. On the other hand, to the best of our understanding, the results from~\cite{DS} do not provide a satisfactory converse result. The difficulty stems from the fact that adapted (Lyapunov) norms that can be used to eliminate the nonuniform part of the exponential dichotomy of the time rescaled system may fail to generate adapted norms for the original system.

 In order to eliminate this difficulty, in the present paper we propose a new approach which does not rely on Lyapunov norms. In the case of discrete time, we show that a nonuniform dichotomy with a slow growth rate can be characterized in terms of a newly introduced notion of a nonuniform ordinary dichotomy with a certain scale (for the original system) together with a nonuniform exponential dichotomy of an associated time rescaled system. In the particular case of invertible systems, the use of nonuniform ordinary dichotomy can be eliminated.

We also establish corresponding results for the case of continuous time by dealing with arbitrary evolution families. The  results obtained in the case of continuous time are substantially stronger. Namely, we are able to characterize general nonuniform dichotomies with growth rates via exponential dichotomies of a time-rescaled evolution family for a large class of growth rates. Thus, we do not need the notion of a nonuniform ordinary dichotomy on a certain scale. In addition, we are not restricted to the case of slowly growing rates. The reason why the case of discrete time is harder lies in the fact that discrete growth rates do not, in general,  map integers to integers, making the construction of the time rescaled system more complicated.

Exploiting the simplicity of our results in the case of continuous time, we give a number of applications. More precisely, we discuss the robustness property of general nonuniform dichotomies, as well as their characterization in terms of admissibility and Lyapunov functions. We also provide new results related to the smooth linearization and stable manifolds of nonautonomous and nonlinear dynamics under the assumption that the linear part admits a general nonuniform dichotomy. Finally, we show that it is essentially sufficient to study the Sacker-Sell spectrum defined with respect to the notion of a (nonuniform) exponential dichotomy, as the general case can be reduced to this one by using time rescaling.

Most of the results obtained in these applications are new, although (as we point out in remarks) similar results appeared in the literature.  However, previous approaches established such result directly, while we show that one can essentially read them from the corresponding results dealing with exponential dichotomies.
\section{The case of discrete time}
	Throughout this paper $X=(X, \| \cdot \|)$ is an arbitrary Banach space. By $\mathcal B(X)$ we denote the space of all bounded linear operators on $X$ equipped with the operator norm, which we also denote by $\| \cdot \|$.
	Let $\{A_n\}_{n \in \mathbb{N}} \subset \mathcal{B}(X)$ and consider the nonautonomous linear system
	\begin{equation}\label{eq:Lin}
	x_{n+1}=A_n x_n, \ \ n \in \mathbb{N}. \tag{A}
	\end{equation}
	Let $\Phi_A=\left\{\Phi_A(m, k)\right\}_{m \ge k}$ be the evolution family associated with~\eqref{eq:Lin} defined by
	\begin{equation}\label{eq:evolution-family-Lin}
	\Phi_A(m, k)
	=\left\{\begin{array}{ll}
		A_{m-1} \cdots A_k, & m>k \\
		\Id, & m=k.
	\end{array} \right.
	\end{equation}
	Here $\Id$ denotes the identity operator on $X$. If operators $A_n$ are invertible we can define $\Phi_A(m, k):=\Phi(k, m)^{-1}$ for $k<m$.
We say that a sequence $\mu=(\mu_n)_{n \in \N_0}$ is a discrete growth rate if it is positive, strictly increasing, and satisfies $\lim\limits_{n\to \infty} \mu_n = +\infty$ and $\mu_0=1$. 

We now recall a construction introduced in~\cite{DS}.
Given a discrete growth rate $\mu=(\mu_n)_{n \in \mathbb{N}_0}$, we can associate to it the continuous, strictly increasing function $\widetilde \mu: [0,+\infty) \to [1,+\infty )$ given by
\[
\widetilde \mu(t)=
\begin{cases}
	\mu_n & \text{ if } \ t=n\\
	\mu_n+(t-n)(\mu_{n+1}-\mu_n) & \text{ if } \ n<t<n+1
\end{cases}, \quad n \in \N_0.
\]
Clearly, this is an invertible function and its inverse ${\widetilde \mu}^{-1}:[1,+\infty) \to [0,+\infty)$ is the continuous, strictly increasing function given by
\[
\widetilde \mu^{-1}(t)=
\begin{cases}
	n & \text{ if } \ t=\mu_n\\
	n+\dfrac{t-\mu_n}{\mu_{n+1}-\mu_n} & \text{ if } \ \mu_n<t<\mu_{n+1}
\end{cases}, \quad n \in \N_0.
\]

\begin{definition}\label{def:ordinary-dich}
    We say that system~\eqref{eq:Lin} (or the sequence $\{A_n\}_{n\in \N}$) admits a nonuniform  $\mu$-ordinary dichotomy with exponent $c\ge 0$ and scale $M>0$ 
 if there are a sequence of projections $\{P_k\}_{k \in \mathbb{N}}$ on $X$ and $K \geq 1$ such that:
	\begin{enumerate}[(od1)]
		\item\label{def:od1} $A_k P_k=P_{k+1} A_k$, for all $k \in \mathbb{N}$;
		\item\label{def:od3} for every $k \in \mathbb{N}, {A_k}_{\mid}: \operatorname{Ker} P_k \rightarrow \operatorname{Ker} P_{k+1}$ is an isomorphism, where ${A_k}_{\mid}$ denotes the restriction of $A_k$ to $\Ker P_k$;
		\item\label{def:od4} $\left\|\Phi_A(m, k) P_k \right\| \leq K\mu_k^c $, for all $m \ge k$ with $\mu_m/\mu_k \le M$;
		\item\label{def:od5} $\left\|\Phi_A(m, k) (\Id-P_k)\right\| \leq K\mu_k^c$, for all $m\le k$ with $\mu_k/\mu_m\le M$.
	\end{enumerate}
\end{definition}

We remark that~(od\ref{def:od1}) is equivalent to the following requirement:
\begin{equation}\label{eq:compatibility-evol-family}
	\Phi_A(m, n) P_n=P_m \Phi_A(m,n), \quad \text{for all $m \ge  n \ge 1$}
\end{equation}
and~(od\ref{def:od3}) is equivalent to the invertibility of $\Phi_A(m,n)_{\mid}: \operatorname{Ker} P_n \rightarrow \operatorname{Ker} P_m$, for all $m \ge n \ge 1$.
Note that, in~(od\ref{def:od5}), $\Phi_A(m,k)$ denotes the inverse of the restriction of $\Phi_A(k,m)$ to $\Ker P_m$.

\begin{definition}\label{def:mu-dich-norms}
Let $\mu=(\mu_n)_{n\in \N_0}$ be a discrete growth rate.
We say that system~\eqref{eq:Lin} (or the sequence $\{A_n\}_{n\in \N}$) admits a nonuniform $\mu$-dichotomy
with exponents $(\nu,\ep)$ where $\nu>0$ and $\ep \ge 0$ if there is a constant $N \geq 1$ and a sequence of projections $\{P_k\}_{k \in \mathbb{N}}$ on $X$ such that (od\ref{def:od1})--(od\ref{def:od3}) in Definition~\ref{def:ordinary-dich} are satisfied and additionally:
\begin{enumerate}[($\mu$1)]
	\item\label{def:mu1} 
	$\left\|\Phi_A(m, k) P_k \right\| \leq N \left(\frac{\mu_m}{\mu_k}\right)^{-\nu} \mu_k^\ep $, for every  $m \ge k$;
	\item\label{def:mu2} $\left\|\Phi_A(m, k) (\Id-P_k)\right\|\leq N \left(\frac{\mu_k}{\mu_m}\right)^{-\nu} \mu_k^\ep$, for every  $m \le k$.
\end{enumerate}
\end{definition}

\begin{remark}
When, in Definition~\ref{def:mu-dich-norms}, we set $\mu_n=e^n$, for $n \in \N_0$, we obtain the notion of nonuniform exponential dichotomy.
Similarly, when we let $\mu_n=n+1$ for $n \in \N_0$, we have the notion of nonuniform polynomial dichotomy. 
When $\ep=0$ we recover the notion of (uniform) $\mu$-dichotomy.
\end{remark}

\begin{remark}
Suppose that the system~\eqref{eq:Lin} admits a nonuniform $\mu$-dichotomy with exponents $(\nu, \varepsilon)$. Then~\eqref{eq:Lin} admits a nonuniform $\mu$-ordinary dichotomy with exponent $\varepsilon$ at any scale $M>0$.
\end{remark}
Following~\cite{DS}, 
given discrete growth rates $\mu=(\mu_n)_{n \in\N_0}$ and $\eta=(\eta_n)_{n \in\N_0}$, we associate with the system~\eqref{eq:Lin} the family of linear operators $Q^{\mu,\eta}=\{Q^{\mu,\eta}_n\}_{n \in \mathbb{N}} \subset \mathcal{B}(X)$, given by 
\begin{equation}\label{eq:Q-mu-eta}
Q^{\mu,\eta}_n=\Phi_A\left(\lfloor \widetilde\mu^{-1}(\eta_n) \rfloor +1, \  \lfloor \widetilde\mu^{-1}(\eta_{n-1}) \rfloor +1 \right), \quad \forall n \in \mathbb{N}.
\end{equation}
Consider the nonautonomous linear system
\begin{equation}\label{eq:Lin-Q}
	y_{n+1}=Q^{\mu,\eta}_n \, y_n, \quad n \in \mathbb{N}. \tag{$Q^{\mu,\eta}$}
\end{equation}
It is easy to obtain the evolution family $\Phi_{Q^{\mu,\eta}}=\left\{\Phi_{Q^{\mu,\eta}}(m,n)\right\}_{m \ge n}$ associated to~\eqref{eq:Lin-Q} in terms of the evolution family $\Phi_A=\left\{\Phi_A(m,n)\right\}_{m \ge n}$ associated to~\eqref{eq:Lin}:
\[
\begin{split}
\Phi_{Q^{\mu,\eta}}(m,n)
& =\prod_{k=n}^{m-1}Q^{\mu,\eta}_k=\prod_{k=n}^{m-1}\Phi_A(\lfloor\widetilde\mu^{-1}(\eta_{k})\rfloor +1,\,\lfloor\widetilde\mu^{-1}(\eta_{k-1})\rfloor +1)\\
& =\Phi_A(\lfloor\widetilde\mu^{-1}(\eta_{m-1})\rfloor +1,\,\lfloor\widetilde\mu^{-1}(\eta_{n-1})\rfloor +1), \quad m\ge n.
\end{split}
\]

In the sequel, given a sequence $\{P_k\}_{k\in \mathbb N}$ of bounded projections on $X$, 
we will also consider the sequence of projections $\{P^\eta_k\}_{k \in \mathbb{N}}$ given by $P^\eta_k=P_{\lfloor \widetilde{\mu}^{-1}(\eta_{k-1})\rfloor +1}$, for all $k \in \mathbb{N}$.

Before we proceed, we present a simple technical lemma that will be useful in the sequel.

\begin{lemma}\label{lem:scale}
	Let $\theta \ge 1$ and $\mu,\eta$ be growth rates satisfying 
	\begin{equation}\label{eq:bound-rates-1-0}
		\frac{\mu_{n+1}}{\mu_n}\le \theta \quad \text{ and } \quad \frac{\eta_{n+1}}{\eta_n}\le \theta, \quad \text{for all} \quad n \in \N_0.
	\end{equation} 
	Then, 
	\begin{enumerate}[(i)]
		\item\label{lem:scale1} for every $r\in\mathbb N_0$, $\eta_r\le \mu_{\lfloor\widetilde\mu^{-1}(\eta_r)\rfloor+1}\le \theta \eta_r$.
		\item\label{lem:scale2} for every $r\in\mathbb N_0$, $\dfrac{\mu_{\lfloor\widetilde\mu^{-1}(\eta_{r+1})\rfloor+1}}{\mu_{\lfloor\widetilde\mu^{-1}(\eta_r)\rfloor+1}} \le \theta^2$.
		\item\label{lem:scale3} if $\lfloor\widetilde\mu^{-1}(\eta_{n+1})\rfloor+1\ge m\ge \lfloor\widetilde\mu^{-1}(\eta_n)\rfloor+1$, then
		$$\dfrac{\mu_m}{\mu_{\lfloor\widetilde\mu^{-1}(\eta_n)\rfloor+1}} \le \theta^2 \quad \text{and} \quad \dfrac{\mu_{\lfloor\widetilde\mu^{-1}(\eta_{n+1})\rfloor+1}}{\mu_m} \le \theta^2$$
	\end{enumerate}
\end{lemma}

\begin{proof}
	Let $r\in\mathbb N_0$, $t=\widetilde{\mu}^{-1}(\eta_r)$. By the definition of $\widetilde{\mu}$ and since it is increasing, we have
	\[
	\eta_r=\tilde\mu(\tilde\mu^{-1}(\eta_r))\le \tilde\mu(\lfloor\tilde\mu^{-1}(\eta_r)\rfloor+1)=\mu_{\lfloor\tilde\mu^{-1}(\eta_r)\rfloor+1}
	\]
	and also, by~\eqref{eq:bound-rates-1-0},
	\[
	\mu_{\lfloor\tilde\mu^{-1}(\eta_r)\rfloor+1}= \frac{\mu_{\lfloor\tilde\mu^{-1}(\eta_r)\rfloor+1}}{\mu_{\lfloor\tilde\mu^{-1}(\eta_r)\rfloor}}\mu_{\lfloor\tilde\mu^{-1}(\eta_r)\rfloor}\le \theta \mu_{\lfloor\tilde\mu^{-1}(\eta_r)\rfloor}=\theta \widetilde{\mu}(\lfloor\widetilde{\mu}^{-1}(\eta_r)\rfloor)\le\theta \widetilde{\mu}(\widetilde{\mu}^{-1}(\eta_r))=\theta\eta_r.
	\]
	This proves \eqref{lem:scale1}. 
	
	Since, by~\eqref{lem:scale1}, we have $\mu_{\lfloor\widetilde\mu^{-1}(\eta_r)\rfloor +1}\ge \eta_r$ and $\mu_{\lfloor\widetilde\mu^{-1}(\eta_{r+1})\rfloor +1}\le \theta \eta_{r+1}$, using~\eqref{eq:bound-rates-1-0} we have
	\[
	\frac{\mu_{\lfloor\widetilde\mu^{-1}(\eta_{r+1})\rfloor+1}}{\mu_{\lfloor\widetilde\mu^{-1}(\eta_r)\rfloor+1}}
	\le \frac{\theta \eta_{r+1}}{\eta_r}
	\le \theta^2. 
	\]
	Thus, we obtain \eqref{lem:scale2}. 
	
	If $\lfloor\widetilde\mu^{-1}(\eta_{n+1})\rfloor+1\ge m\ge \lfloor\widetilde\mu^{-1}(\eta_n)\rfloor+1$, then $\mu_{\lfloor\widetilde\mu^{-1}(\eta_{n+1})\rfloor+1}\ge \mu_m\ge \mu_{\lfloor\widetilde\mu^{-1}(\eta_n)\rfloor+1}$ and thus, from~\eqref{lem:scale2},
	\[
	\frac{\mu_m}{\mu_{\lfloor\widetilde\mu^{-1}(\eta_n)\rfloor+1}}
	\le \frac{\mu_{\lfloor\widetilde\mu^{-1}(\eta_{n+1})\rfloor+1}}{\mu_{\lfloor\widetilde\mu^{-1}(\eta_n)\rfloor+1}}
	\le\theta^2
	\quad \text{ and } \quad
	\frac{\mu_{\lfloor\widetilde\mu^{-1}(\eta_{n+1})\rfloor+1}}{\mu_m}
	\le \frac{\mu_{\lfloor\widetilde\mu^{-1}(\eta_{n+1})\rfloor+1}}{\mu_{\lfloor\widetilde\mu^{-1}(\eta_n)\rfloor+1}}
	\le\theta^2.
	\]
	Hence, \eqref{lem:scale3} holds. 
\end{proof}

The following results relate the existence of a nonuniform $\mu$-dichotomy for system~\eqref{eq:Lin} to that of a nonuniform $\eta$-dichotomy for system~\eqref{eq:Lin-Q}.

\begin{theorem}\label{teo:main1}
    Let $\theta\ge 1$ and $\mu,\eta$ be growth rates satisfying 
	\begin{equation}\label{eq:bound-rates-1-a}
		\frac{\mu_{n+1}}{\mu_n}\le \theta \quad \text{ and } \quad \frac{\eta_{n+1}}{\eta_n}\le \theta, \quad \text{for all} \quad n \in \N_0.
	\end{equation} 
	If system~\eqref{eq:Lin} admits a nonuniform $\mu$-dichotomy 
    relative to a sequence of projections $\{P_k\}_{k \in \N}$, then system~\eqref{eq:Lin-Q} admits a nonuniform $\eta$-dichotomy 
    relative to the sequence of projections $\left\{P^\eta_k\right\}_{k \in \mathbb{N}}$.
\end{theorem}

\begin{proof}
	Assume that the system~\eqref{eq:Lin} admits a nonuniform $\mu$-dichotomy
    relative to a sequence of projections $\{P_k\}_{k \in \mathbb{N}}$ and let $N \geq 1$, $\ep\ge 0$ and $\nu>0$ be given by Definition~\ref{def:mu-dich-norms}. Then, for all $m \ge k$, we have
	\begin{equation}\label{mu1}\left\|\Phi_A(m, k) P_k \right\| \leq N \left(\frac{\mu_m}{\mu_k}\right)^{-\nu}\mu_k^\ep\end{equation}
	and, for all $m \le k$, we have
	\begin{equation}\label{mu2}\left\|\Phi_A(m, k) (\Id-P_k)\right\| \leq N \left(\frac{\mu_k}{\mu_m}\right)^{-\nu}\mu_k^\ep.\end{equation}
	
	
	
	On the other hand, by~\eqref{eq:compatibility-evol-family}, we obtain, for all $k \in \mathbb{N}$,
	\[
	\begin{split}
		Q^{\mu,\eta}_k P^\eta_k 
		& =\Phi_A\left(\lfloor \widetilde\mu^{-1}(\eta_k) \rfloor +1, \  \lfloor \widetilde\mu^{-1}(\eta_{k-1}) \rfloor +1\right) P_{\lfloor \widetilde{\mu}^{-1}(\eta_{k-1})\rfloor +1} \\
		& =P_{\lfloor \widetilde{\mu}^{-1}(\eta_{k})\rfloor +1} \Phi_A\left(\lfloor \widetilde\mu^{-1}(\eta_k) \rfloor +1, \  \lfloor \widetilde\mu^{-1}(\eta_{k-1}) \rfloor +1\right) =P^\eta_{k+1} Q^{\mu,\eta}_k.
	\end{split}
	\]
	
	By~(od\ref{def:od3}) we conclude that, for every $k \in \mathbb{N}$, $$\Phi_A\left(\lfloor \widetilde\mu^{-1}(\eta_k) \rfloor +1, \  \lfloor \widetilde\mu^{-1}(\eta_{k-1}) \rfloor +1\right)_\mid: \operatorname{Ker} P_{\lfloor \widetilde\mu^{-1}(\eta_{k-1}) \rfloor +1} \rightarrow \operatorname{Ker} P_{\lfloor \widetilde\mu^{-1}(\eta_k) \rfloor +1}$$ 
	is invertible. In other words, ${Q^{\mu,\eta}_k}_{\mid}: \operatorname{Ker} P^\eta_k \rightarrow \operatorname{Ker} P^\eta_{k+1}$ is invertible.
	
    Let $k \in \mathbb{N}$. Using Lemma~\ref{lem:scale} and \eqref{eq:bound-rates-1-a}, we obtain, for $m,k \in \N$,
    \begin{equation}\label{eq:bound3}
    \frac{\mu_{\lfloor\widetilde\mu^{-1}(\eta_{m-1})\rfloor +1}}{\mu_{\lfloor\widetilde\mu^{-1}(\eta_{k-1})\rfloor +1}} \ge \frac{\eta_{m-1}}{\theta\eta_{k-1}} \ge \frac{\eta_{m}}{\theta^2\eta_k}.
    \end{equation}
    Consequently, for $m\ge k$, using \eqref{eq:bound-rates-1-a}, \eqref{mu1} and \eqref{eq:bound3},
    \[
	\begin{split}
		\left\|\Phi_{Q^{\mu,\eta}}(m,k)P_k^\eta \right\|
		& =\left\|\Phi_A(\lfloor\widetilde\mu^{-1}(\eta_{m-1})\rfloor +1,\,\lfloor\widetilde\mu^{-1}(\eta_{k-1})\rfloor+1)P_{\lfloor\widetilde\mu^{-1}(\eta_{k-1})\rfloor +1 }\right \|\\
		& \leq N \left(\frac{\mu_{\lfloor\widetilde\mu^{-1}(\eta_{m-1})\rfloor +1}}{\mu_{\lfloor\widetilde\mu^{-1}(\eta_{k-1})\rfloor +1}}\right)^{-\nu}
		(\mu_{\lfloor\widetilde\mu^{-1}(\eta_{k-1})\rfloor +1})^\ep \\
		& \leq N\theta^{2\nu+\ep} \left(\frac{\eta_{m}}{\eta_k}\right)^{-\nu}\eta_{k}^\ep.
	\end{split}
	\]
    Similarly, for $1\le m\le k$, using \eqref{eq:bound-rates-1-a}, \eqref{mu2} and \eqref{eq:bound3},
    \[
	\begin{split}
		\left\|\Phi_{Q^{\mu,\eta}}(m,k)(\Id-P^\eta_k)\right\|
		& =\left\|\Phi_A(\lfloor\widetilde\mu^{-1}(\eta_{m-1})\rfloor +1,\,\lfloor\widetilde\mu^{-1}(\eta_{k-1})\rfloor+1)(\Id-P_{\lfloor\widetilde\mu^{-1}(\eta_{k-1})\rfloor+1})\right\|\\
		& \leq N \left(\frac{\mu_{\lfloor\widetilde\mu^{-1}(\eta_{k-1})\rfloor+1}}{\mu_{\lfloor\widetilde\mu^{-1}(\eta_{m-1})\rfloor+1}}\right)^{-\nu}
		(\mu_{\lfloor\widetilde 
        \mu^{-1}(\eta_{k-1})\rfloor+1})^\ep \\
	    & \leq N\theta^{2\nu+\ep} \left(\frac{\eta_k}{\eta_m} \right)^{-\nu} \eta_k^\ep.	 
    \end{split}
    \]

It follows that~\eqref{eq:Lin-Q} admits an $\eta$-dichotomy 
relative to the sequence of projections $\left\{P^\eta_k\right\}_{k \in \mathbb{N}}$. The proof of the theorem is completed.
\end{proof}

\begin{theorem}\label{teo:main}
Let $\theta \ge 1$ and $\mu, \eta$ be growth rates satisfying
\begin{equation}\label{eq:bound-rates-1-b}
	\frac{\mu_{n+1}}{\mu_n}\le \theta \quad \text{and} \quad \frac{\eta_{n+1}}{\eta_n} \le \theta, \quad \text{for all} \quad n \in \N_0.
\end{equation} 
If system~\eqref{eq:Lin} admits a nonuniform $\mu$-ordinary dichotomy 
relative to a sequence of projections $\{P_k\}_{k \in \mathbb{N}}$, with exponent $c$ and scale $\theta^2$, and system~\eqref{eq:Lin-Q} admits a nonuniform $\eta$-dichotomy 
relative to the sequence of projections $\left\{P^\eta_k\right\}_{k \in \mathbb{N}}$, with exponents $(\nu,\ep)$, then, if $\nu>c$, system~\eqref{eq:Lin} admits a nonuniform $\mu$-dichotomy 
relative to a sequence of projections $\{P_k\}_{k \in \mathbb{N}}$, with exponents $(\nu-c,\ep+2c)$.

\end{theorem}

\begin{proof}
Assume that~\eqref{eq:Lin} admits a nonuniform $\mu$-ordinary dichotomy 
relative to a sequence of projections $\{P_k\}_{k \in \mathbb{N}}$ with exponent $c$ at scale $M:=\theta^2$: for some $K \geq 1$  we have, for all  $m \ge k$ with $\mu_m/\mu_k \le M$,
\begin{equation}\label{eq:ordinary-dich1-iii->i}
\left\|\Phi_A(m, k) P_k \right\|\leq K\mu_k^c
\end{equation}
and, for all $m \le k$ with $\mu_k/\mu_m\le M$,
\begin{equation}\label{eq:ordinary-dich2-iii->i}
\left\|\Phi_A(m, k) (\Id-P_k)\right\| \leq K\mu_k^c. 
\end{equation}

Assume also that system~\eqref{eq:Lin-Q} admits a nonuniform $\eta$-dichotomy 
relative to the projections $\left\{P^\eta_k\right\}_{k \in \mathbb{N}}$: there are $N \geq 1, \nu>0$ and  $\ep \ge 0$ such that, for all $m \ge k$, we have
\begin{equation}\label{eq:eta-dich1-iii->i}
\left\|\Phi_{Q^{\mu,\eta}}(m,k) P_k^\eta \right\| \leq N \left(\frac{\eta_m}{\eta_k}\right)^{-\nu}\eta_k^\ep
\end{equation}
and for $m\le k$,
\begin{equation}\label{eq:eta-dich2-iii->i}
\left\|\Phi_{Q^{\mu,\eta}}(m,k) (\Id-P_k^\eta)\right\| \leq N \left(\frac{\eta_k}{\eta_m}\right)^{-\nu}\eta_k^\ep.
\end{equation}
We claim that there exists $N_1\ge 1$ such that  for $m \ge k \ge 1$, we have
\begin{equation}\label{eq:claim1}
	\left\|\Phi_A(m, k) P_k \right\|
	\le N_1 \left(\frac{\mu_{m}}{\mu_{k}}\right)^{-\nu+c} \mu_k^{\ep+2c}.
\end{equation}
We split the proof of~\eqref{eq:claim1} into two cases:\\
\underline{Case I} - There are $n,j \in \N_0$ with $n>j$ such that 
$$\lfloor\widetilde\mu^{-1}(\eta_{n+1})\rfloor +1> m \ge \lfloor\widetilde\mu^{-1}(\eta_n)\rfloor +1 \ge \lfloor\widetilde\mu^{-1}(\eta_{j+1})\rfloor +1 >  k
\ge \lfloor\widetilde\mu^{-1}(\eta_j)\rfloor+1.$$
 Recalling that $P_{\lfloor {\widetilde \mu}^{-1} (\eta_{j+1}) \rfloor +1}=P_{j+2}^\eta$ 
, by~\eqref{eq:bound-rates-1-b},~\eqref{eq:ordinary-dich1-iii->i},~\eqref{eq:eta-dich1-iii->i} and~\eqref{lem:scale3} in Lemma~\ref{lem:scale}, we get 
\[
\begin{split}
& \left\|\Phi_A(m, k) P_k \right\| \\
	& =\left\|
	\Phi_A\left(m, \lfloor\widetilde\mu^{-1}(\eta_n)\rfloor +1\right) P_{\lfloor \widetilde{\mu}^{-1}(\eta_n)\rfloor+1} \Phi_A\left(\lfloor\widetilde\mu^{-1}(\eta_n)\rfloor +1, k\right) P_k \right\| \\
	& \le K \mu_{\lfloor\widetilde\mu^{-1}(\eta_n)\rfloor +1}^c \left\| \Phi_A\left(\lfloor\widetilde\mu^{-1}(\eta_n)\rfloor +1, \lfloor\widetilde\mu^{-1}(\eta_{j+1})\rfloor +1\right)
	\Phi_A\left(\lfloor\widetilde\mu^{-1}(\eta_{j+1})\rfloor +1,k\right) P_k \right\|\\
&\le K\mu_m^c\left\|\Phi_{Q^{\mu,\eta}}\left(n+1,j+2\right)P_{j+2}^\eta \Phi_A\left(\lfloor\widetilde\mu^{-1}(\eta_{j+1})\rfloor+1,k\right) P_k \right\| \\
	& \leq NK  \mu_m^c\left(\frac{\eta_{n+1}}{\eta_{j+2}}\right)^{-\nu}
	\eta_{j+2}^\ep \, \left\|\Phi_A\left(\lfloor\widetilde\mu^{-1}(\eta_{j+1})\rfloor+1,k\right) P_k\right\| \\
	& \leq NK^2   \left(\frac{\mu_m}{\mu_k}\right)^c\left(\frac{\eta_{n+1}}{\eta_{j+2}}\right)^{-\nu}
	\eta_{j+2}^\ep\mu_k^{2c}\\
	& \leq NK^2 \theta^{3\nu+2\ep}  \left(\frac{\mu_m}{\mu_k}\right)^c\left(\frac{\mu_m}{\mu_k}\right)^{-\nu}
	\mu_k^\ep\mu_k^{2c}\\
	& \leq NK^2 \theta^{3\nu+2\ep} \left(\frac{\mu_m}{\mu_k}\right)^{-\nu+c} \mu_k^{\ep+2c},
\end{split}
\]
as 
\[
\begin{split}
\frac{\eta_{n+1}}{\eta_{j+2}}=\frac{\eta_{n+1}}{\mu_m}\cdot \frac{\mu_m}{\mu_k}\cdot \frac{\mu_k}{\eta_{j+2}} &\ge \frac{\eta_{n+1}}{\mu_{\lfloor \widetilde{\mu}^{-1}(\eta_{n+1})\rfloor+1}}\cdot \frac{\mu_m}{\mu_k}\cdot \frac{\mu_{\lfloor \widetilde{\mu}^{-1}(\eta_j)\rfloor+1}}{\eta_{j+2}}\\
&\ge \frac{\eta_{n+1}}{\mu_{\lfloor \widetilde{\mu}^{-1}(\eta_{n+1})\rfloor+1}}\cdot \frac{\mu_m}{\mu_k}\cdot \frac{\eta_j}{\eta_{j+2}}\\
&\ge \theta^{-3}\frac{\mu_m}{\mu_k}.
\end{split}
\]

\underline{Case II} - There is $n \in \N_0$ such that $$\lfloor\widetilde\mu^{-1}(\eta_{n+1})\rfloor +1> m \ge k
\ge \lfloor\widetilde\mu^{-1}(\eta_n)\rfloor+1.$$
By~\eqref{eq:ordinary-dich1-iii->i} and~\eqref{lem:scale2} in Lemma~\ref{lem:scale}, we get 
\[
\begin{split}
	\left\|\Phi_A(m, k) P_k \right\|	\le K\mu_k^c 
	&=K\mu_k^c\left (\frac{\mu_m}{\mu_k}\right )^{-\nu} \left (\frac{\mu_m}{\mu_k}\right )^{\nu}  \\
&\le K\mu_k^c\left (\frac{\mu_m}{\mu_k}\right )^{-\nu} \left (\frac{\mu_{\lfloor\tilde\mu^{-1}(\eta_{n+1})\rfloor+1}}{\mu_{\lfloor\tilde\mu^{-1}(\eta_n)\rfloor+1}}\right )^{\nu}\\
& \le K \theta^{2\nu}
\left(\frac{\mu_{m}}{\mu_{k}}\right)^{-\nu}\mu_k^c.
\end{split}
\]
We conclude that~\eqref{eq:claim1} holds and the claim follows. 

Our next claim is that there exists $N_2\ge 1$ such that for $1\le m\le k$, we have
\begin{equation}\label{eq:claim2}
	\left\|\Phi_A(m, k) (\Id-P_k) \right\|
	\leq N_2 \left(\frac{\mu_{k}}{\mu_{m}}\right)^{-\nu}   \mu_k^{\ep+2c}.
\end{equation}
We again split the proof of~\eqref{eq:claim2} into two cases.

\underline{Case I} - There are $n, j\in \N_0$ with $n<j$ such that 
\[
\lfloor \tilde \mu^{-1}(\eta_n)\rfloor +1\le m < \lfloor \tilde \mu^{-1}(\eta_{n+1})\rfloor +1 \le  \lfloor \tilde \mu^{-1}(\eta_{j})\rfloor +1\le k< \lfloor \tilde{\mu}^{-1}(\eta_{j+1})\rfloor +1.
\]
By~\eqref{eq:bound-rates-1-b},~\eqref{eq:ordinary-dich2-iii->i},~\eqref{eq:eta-dich2-iii->i} and~\eqref{lem:scale3} in Lemma~\ref{lem:scale} we have that 
\[
\begin{split}
& \quad \left\|\Phi_A(m, k) (\Id-P_k)\right\|\\ 
	& =\left\|\Phi_A\left(m, \lfloor\widetilde\mu^{-1}(\eta_{n+1})\rfloor+1\right) (\Id- P_{\lfloor\widetilde\mu^{-1}(\eta_{n+1})\rfloor+1})\Phi_A\left(\lfloor\widetilde\mu^{-1}(\eta_{n+1})\rfloor +1, k\right) (\Id-P_k)\right\|\\
 &\le K\mu_{\lfloor\widetilde\mu^{-1}(\eta_{n+1})\rfloor+1}^c \left \|\Phi_A\left(\lfloor\widetilde\mu^{-1}(\eta_{n+1})\rfloor+1, \lfloor\widetilde\mu^{-1}(\eta_{j})\rfloor +1\right)
	\Phi_A\left(\lfloor\widetilde\mu^{-1}(\eta_{j})\rfloor+1,k\right) (\Id-P_k)\right \|\\
&\le K^2\theta^{2c}\mu_k^c\mu_m^c \|\Phi_A\left(\lfloor\widetilde\mu^{-1}(\eta_{n+1})\rfloor+1, \lfloor\widetilde\mu^{-1}(\eta_{j})\rfloor +1\right)(\Id-P_{\lfloor\widetilde\mu^{-1}(\eta_{j})\rfloor +1})\|\\
&\le K^2\theta^{2c}\mu_k^c\mu_m^c \|\Phi_{Q^{\mu,\eta}}\left(n+2,j+1\right)(\Id-P_{j+1}^\eta)\|\\
& \le NK^2\theta^{2c}\mu_k^{2c}\left(\frac{\mu_m}{\mu_k}\right)^c \left(\frac{\eta_{j+1}}{\eta_{n+2}}\right)^{-\nu}\theta^\ep\mu_k^\ep\\
&\le NK^2 \theta^{2c+3\nu+\ep}\left (\frac{\mu_k}{\mu_m}\right )^{-(\nu+c)}\mu_k^{\ep+2c},
\end{split} 
\]
as \[\frac{\eta_{j+1}}{\mu_k}\le \frac{\eta_{j+1}}{\mu_{\lfloor \widetilde{\mu}^{-1}(\eta_j)\rfloor+1}}\le \frac{\eta_{j+1}}{\eta_j}\le \theta, \]
and
\[
\frac{\eta_{j+1}}{\eta_{n+2}}=\frac{\eta_{j+1}}{\mu_k}\cdot  \frac{\mu_m}{\eta_{n+2}}\cdot \frac{\mu_k}{\mu_m}\ge \frac{\eta_{j+1}}{\mu_{\lfloor \widetilde{\mu}^{-1}(\eta_{j+1})\rfloor+1}}\cdot \frac{\mu_{\lfloor \widetilde{\mu}^{-1}(\eta_n)\rfloor+1}}{\eta_{n+2}}\cdot \frac{\mu_k}{\mu_m}\ge \theta^{-1}\frac{\eta_n}{\eta_{n+2}}\cdot \frac{\mu_k}{\mu_m}\ge \theta^{-3}\frac{\mu_k}{\mu_m}.
\]

\underline{Case II} - There is $n\in \N_0$ such that
\[
\lfloor \tilde{\mu}^{-1}(\eta_n)\rfloor +1 \le m\le k<\lfloor \tilde{\mu}^{-1}(\eta_{n+1})\rfloor+1.
\]
By~~\eqref{lem:scale2} in Lemma~\ref{lem:scale}
and~\eqref{eq:ordinary-dich2-iii->i} we have that 
\[
\begin{split}
\left\|\Phi_A(m, k) (\Id-P_k)\right\| \le K\mu_k^c &=K\left (\frac{\mu_k}{\mu_m}\right )^{-\nu} \left (\frac{\mu_k}{\mu_m}\right )^{\nu} \mu_k^c \\
&\le K\left (\frac{\mu_k}{\mu_m}\right )^{-\nu} \left (\frac{\mu_{\lfloor\tilde\mu^{-1}(\eta_{n+1})\rfloor+1}}{\mu_{\lfloor\tilde\mu^{-1}(\eta_n)\rfloor+1}}\right )^{\nu} \mu_k^c \\
&\le K\theta^{2\nu} \left (\frac{\mu_k}{\mu_m}\right )^{-\nu}\mu_k^c.
\end{split}
\]
We conclude that~\eqref{eq:claim2} holds and the claim follows.

We conclude that system~\eqref{eq:Lin} admits a $\mu$-nonuniform dichotomy 
relative to the projections $\{P_k\}_{k \in \N}$ with exponents $(\nu-c,\ep+2c)$. The proof of the theorem is completed.
\end{proof}

Before giving conditions under which the ordinary dichotomy assumption in Theorem~\ref{teo:main} follows automatically, we present a simple example showing that this theorem applies in a nonuniform setting in which no uniform exponential dichotomy exists. More precisely, the system below admits the required local ordinary dichotomy and a nonuniform exponential dichotomy, but it does not admit any uniform exponential dichotomy.

\begin{example}\label{ex:discrete}
We equip $\R^2$ with the maximum norm. Define a sequence $(b_n)_{n\in\mathbb N_0}$ by
\begin{equation}\label{eq:discrete-b}
 b_n=
 \begin{cases}
  n-3\cdot 2^{j-1}, & 2^j\leq n\leq 3\cdot 2^{j-1}\text{ for some }j\in\N,\\[1mm]
  0, & \text{otherwise}.
 \end{cases}
\end{equation}
Note that the intervals in~\eqref{eq:discrete-b} are pairwise disjoint. Consider the system
\begin{equation}\label{eq:discrete-system}
 x_{n+1}=A_nx_n,
 \qquad
 A_n=
 \begin{pmatrix}
  \e^{-1+b_{n+1}-b_n} & 0\\
  0 & \e^{1-b_{n+1}+b_n}
 \end{pmatrix}.
\end{equation}

The evolution operators associated with~\eqref{eq:discrete-system} are
\begin{equation}\label{eq:discrete-evolution}
 \Phi_A(m,k)=
 \begin{pmatrix}
  \e^{-(m-k)+b_m-b_k} & 0\\
  0 & \e^{m-k-b_m+b_k}
 \end{pmatrix},
 \qquad m,k\in\mathbb N_0.
\end{equation}
For $m<k$, the expression in~\eqref{eq:discrete-evolution} denotes the inverse of $\Phi_A(k,m)$.  Let
\[
 P=
 \begin{pmatrix}
  1&0\\
  0&0
 \end{pmatrix}.
\]
Clearly, $A_nP=PA_n$ for every $n$, and the restriction of $A_n$ to $\ker P$ is an isomorphism.

We highlight two elementary properties of $(b_n)$.  Namely,
\begin{equation}\label{eq:discrete-bounds}
 -\frac n2\leq b_n\leq0 \qquad\text{and}\qquad b_{n+1}-b_n\leq1 \quad\text{for every } n\in\mathbb N_0.
\end{equation}
Indeed, if $n=2^j+r$ with $0\leq r\leq 2^{j-1}$, then
$$b_n=-\frac{2^j}{2}+r\geq-\frac{2^j}{2}\ge -\frac n2.$$
When $b_n=0$, we immediately have $b_n\ge -n/2$.  To obtain the second inequality in~\eqref{eq:discrete-bounds}, we note the following: when $m$ and $m+1$ belong to the same interval, $[2\cdot 2^{j-1},3\cdot 2^{j-1}]$, we have $b_{m+1}-b_m=1$; in all the remaining cases $b_{m+1}-b_m\le 0$.

It follows from~\eqref{eq:discrete-bounds} that
\begin{equation}\label{eq:discrete-one-sided-lipschitz}
 b_m-b_k\leq m-k \qquad\text{whenever }m\geq k.
\end{equation}
Consequently, for $m\geq k$,
$$\|\Phi_A(m,k)P\|=\e^{-(m-k)+b_m-b_k}\leq1.$$
For $m\leq k$, applying~\eqref{eq:discrete-one-sided-lipschitz} to the pair $(k,m)$ gives
$$\|\Phi_A(m,k)(\Id-P)\| =\e^{-(k-m)+b_k-b_m}\leq1.$$
Thus system~\eqref{eq:discrete-system} admits a nonuniform exponential ordinary dichotomy with $c=0$ and $K=1$. In fact, these estimates hold without any restriction on
$|m-k|$, so the scale $M>1$ can be chosen arbitrarily.

If $m\geq k$, we have $b_m\leq0$ and $b_k\geq-k/2$, and hence 
\begin{equation}\label{eq:discrete-stable}
 \|\Phi_A(m,k)P\|
 \leq \e^{-(m-k)}\e^{k/2}.
\end{equation}
If $m\leq k$, then $b_k\leq0$ and $b_m\geq-m/2$, and therefore
\begin{equation}\label{eq:discrete-unstable}
 \begin{split}
 \|\Phi_A(m,k)(\Id-P)\|
 &=\e^{-(k-m)-b_m+b_k}\\
 &\leq \e^{-(k-m)}\e^{m/2}
 \leq \e^{-(k-m)}\e^{k/2}.
 \end{split}
\end{equation}
Equations~\eqref{eq:discrete-stable} and~\eqref{eq:discrete-unstable} prove that system~\eqref{eq:discrete-system} admits a nonuniform exponential dichotomy with
$N=1$, $\nu=1$, and $\varepsilon=1/2$.  Notice in particular that $\varepsilon<\nu$.

We claim that system~\eqref{eq:discrete-system} does not admit a uniform exponential dichotomy, even if a different sequence of projections is allowed. In fact, for
$$2^j\leq n<3\cdot 2^{j-1}$$
we have $b_{n+1}-b_n=1$, and thus $A_n=\Id$.  Hence
\begin{equation}\label{eq:discrete-identity-block}
 \Phi_A(3\cdot 2^{j-1},2^j)=\Id, \quad \text{with} \quad
 3\cdot2^{j-1}-2^j=2^{j-1}\longrightarrow+\infty.
\end{equation}

Suppose, to the contrary, that the system admitted a uniform exponential dichotomy with invariant projections $(R_n)$ and constants $D\geq1$ and $\alpha>0$.  Invariance on the identity block gives $R_{3\cdot 2^{j-1}}=R_{2^j}$. The stable and unstable estimates, applied to the two endpoints of this
block, would imply
$$\|R_{2^j}\|\leq D\e^{-\alpha 2^{j-1}} \qquad\text{and}\qquad \|\Id-R_{2^j}\|=\|\Id-R_{3\cdot 2^{j-1}}\|\leq D\e^{-\alpha 2^{j-1}}.$$
Therefore,
$$1=\|\Id\| \leq \|R_{2^j}\|+\|\Id-R_{2^j}\| \leq2D\e^{-\alpha 2^{j-1}},$$
which is impossible as $j\to\infty$.  This proves the claim.

Finally, for $\mu_n=\eta_n=e^n$, one has $\theta=e$ and $Q_n^{\mu,\eta}=A_n$. The ordinary dichotomy obtained above holds at every scale and hence, in particular, at the scale $\theta^2=e^2$. Moreover, $\nu=1>0=c$. Thus, all the hypotheses of Theorem~\ref{teo:main} are satisfied.
\end{example}

In the following result we formulate appropriate assumptions under which we can remove the assumptions regarding the existence of nonuniform $\mu$-ordinary dichotomy in Theorem~\ref{teo:main}.
\begin{theorem}\label{theomain3}
Let $\theta\ge 1$ and $\mu, \eta$ be growth rates satisfying~\eqref{eq:bound-rates-1-b}.
Suppose that $A_n$ is invertible for each $n\in \mathbb N$ and that there exist $C, \omega>0$ and $\delta \ge 0$ such that
\begin{equation}\label{bgrow}
\|\Phi_A(m, k)\| \le C\left (\frac{\mu_m}{\mu_k}\right )^\omega \, \mu_k^\delta  \quad \text{for $m\ge k$,}
\end{equation}
and
\begin{equation}\label{bdec}
\|\Phi_A(m, k)\| \le C\left (\frac{\mu_k}{\mu_m}\right )^\omega \, \mu_k^\delta  \quad \text{for $m\le k$.}
\end{equation}
Furthermore, assume that system~\eqref{eq:Lin-Q} admits a nonuniform $\eta$-dichotomy 
with exponents $(\nu,\ep)$ such that $\nu>3\delta+\varepsilon$.
Then~\eqref{eq:Lin} admits a nonuniform $\mu$-dichotomy with exponents $(\nu-\ep-3\delta, 6\delta+3\ep)$.
\end{theorem}

\begin{proof}
Since~\eqref{eq:Lin-Q} admits a nonuniform $\eta$-dichotomy with exponents $(\nu, \ep)$ and relative to a sequence of projections $\{\tilde P_k\}_{k\in \mathbb N}$, there exists $N\ge 1$ such that 
\begin{equation}\label{959-01}
\left\|\Phi_{Q^{\mu,\eta}}(m,k) \tilde P_k\right\| \leq N \left(\frac{\eta_m}{\eta_k}\right)^{-\nu}\eta_k^\ep \quad \text{for $m\ge k$,}
\end{equation}
and
\[
\left\|\Phi_{Q^{\mu,\eta}}(m,k) (\Id-\tilde P_k)\right\| \leq N \left(\frac{\eta_k}{\eta_m}\right)^{-\nu}\eta_k^\ep \quad \text{for $m\le k$.}
\]
From~\eqref{959-01} applied to $m=k$ we conclude that 
\begin{equation}\label{334}
\|\tilde P_k\| \le N\eta_k^\varepsilon, \quad \text{for $k\in \mathbb N$.}
\end{equation}
We now construct a new sequence of projections $\{P_k\}_{k\in \mathbb N}$ in the following manner: fix $k\in \N$ and choose  $j\in \mathbb N$ such that 
\[
\lfloor \widetilde{\mu}^{-1}(\eta_{j})\rfloor+1>k\ge \lfloor \widetilde{\mu}^{-1}(\eta_{j-1})\rfloor+1.
\]
Set 
\[
P_k:=\Phi_A (k, \lfloor \widetilde{\mu}^{-1}(\eta_{j-1})\rfloor+1)  \tilde P_j\Phi_A(\lfloor \widetilde{\mu}^{-1}(\eta_{j-1})\rfloor+1, k).
\]
It is easy to verify that~(od1) holds. Indeed, if 
\[
\lfloor \widetilde{\mu}^{-1}(\eta_{j})\rfloor+1>k+1\ge \lfloor \widetilde{\mu}^{-1}(\eta_{j-1})\rfloor+1,
\]
we have  \[
\begin{split}
A_kP_k&=A_k\Phi_A (k, \lfloor \widetilde{\mu}^{-1}(\eta_{j-1})\rfloor+1)  \tilde P_j\Phi_A(\lfloor \widetilde{\mu}^{-1}(\eta_{j-1})\rfloor+1, k) \\
&=\Phi_A (k+1, \lfloor \widetilde{\mu}^{-1}(\eta_{j-1})\rfloor+1)  \tilde P_j\Phi_A(\lfloor \widetilde{\mu}^{-1}(\eta_{j-1})\rfloor+1, k) \\
&=\Phi_A (k+1, \lfloor \widetilde{\mu}^{-1}(\eta_{j-1})\rfloor+1)  \tilde P_j\Phi_A(\lfloor \widetilde{\mu}^{-1}(\eta_{j-1})\rfloor+1, k+1)A_k  \\
&=P_{k+1}A_k.
\end{split}
\]
Assume now that $k+1=\lfloor \widetilde{\mu}^{-1}(\eta_{j})\rfloor+1,
$
and choose  $m$ such that 
\[
\lfloor \widetilde{\mu}^{-1}(\eta_{m})\rfloor+1>k+1\ge \lfloor \widetilde{\mu}^{-1}(\eta_{m-1})\rfloor+1,
\]
Then we have 
\[
\begin{split}
P_{k+1}&= \Phi_A (k+1, \lfloor \widetilde{\mu}^{-1}(\eta_{m-1})\rfloor+1)  \tilde P_m\Phi_A(\lfloor \widetilde{\mu}^{-1}(\eta_{m-1})\rfloor+1, k+1) \\
&=\Phi_{Q^{\mu, \eta}}(j+1, m)\tilde P_m\Phi_{Q^{\mu, \eta}}(m, j+1)\\
&=\tilde P_{j+1}.
\end{split}
\]
Therefore,
\[
A_kP_k=\Phi_{Q^{\mu, \eta}}(j+1, j)\tilde P_j\Phi_{Q^{\mu, \eta}}(j, j+1)A_k=\tilde P_{j+1}A_k=P_{k+1}A_k,
\]
yielding~(od1). Note that due to invertibility, (od2) is trivially satisfied.

By Lemma~\ref{lem:scale}(ii), \eqref{eq:bound-rates-1-b}, \eqref{bgrow}, \eqref{bdec} and~\eqref{334} we have 
\begin{equation}\label{341}
\|P_k\| \le C^2N \theta^{4\omega+\ep} 
\mu_k^{2\delta+\ep}, 
\end{equation}
for $k\in \mathbb N$. 

Next, we note that~\eqref{bgrow} and~\eqref{341} imply that 
\[
\|\Phi_A(m, k)P_k\| \le \|\Phi_A(m, k)\|\cdot \|P_k\| \le K\mu_k^{3\delta+\varepsilon},
\]
for $m\ge k$ such that $\mu_m/\mu_k \le \theta^2$, where $K>0$ is independent of $m$ and $k$. 
Similar estimate can be established along the complementary direction. We conclude that~\eqref{eq:Lin} admits a nonuniform $\mu$-ordinary dichotomy with exponent $c=3\delta+\varepsilon$ and scale $\theta^2$ relative to the sequence of projections $\{P_k\}_{k\in \mathbb N}$.  Observe that 
\[
P_k^\eta=\tilde P_k, \quad k\in \N.
\]
The conclusion of the theorem now follows directly from Theorem~\ref{teo:main}.
\end{proof}

\section{The case of continuous time}
We recall that $\{T(t, s)\}_{t\ge s \ge 0}\subset \mathcal B(X)$ is said to be an \emph{evolution family} if the following holds:
\begin{itemize}
\item $T(t, t)=\Id$ for $t\ge 0$;
\item $T(t, s)T(s, r)=T(t, r)$ for $t\ge s\ge r \ge 0$;
\item for $v\in X$ and $s\ge 0$, the map $t \mapsto T(t, s)v$ is continuous on $[s, \infty)$.
\end{itemize}
A (continuous-time)  growth rate is an arbitrary continuous bijective map $\mu \colon [0, \infty)\to [1, \infty)$ such that $\mu(0)=1$.

\begin{definition}
    Let  $\{T(t, s)\}_{t\ge s \ge 0}\subset \mathcal B(X)$  be an evolution family and $\mu \colon [0, \infty)\to [1, \infty)$ a growth rate. We say that $\{T(t, s)\}_{t\ge s \ge 0}$ admits a nonuniform $\mu$-dichotomy with exponents $(\nu, \ep)$ where $\nu>0$ and $\ep\ge 0$ if there is a constant $N\ge 1$ and a family of projections $\{P(t)\}_{t\ge 0}\subset \mathcal B(X)$ such that:
\begin{enumerate}[($\mu c$1)]
    \item\label{def:muc1} 
$T(t, s)P(s)=P(t)T(t, s)$ for $t\ge s\ge 0$;
	\item\label{def:muc2} 
for every $t\ge s\ge 0$, $T(t, s)_{\mid}\colon \Ker P(s)\to \Ker P(t)$ is an isomorphism, where $T(t, s)_{\mid}$ denotes the restriction of $T(t, s)$ to $\Ker P(s)$;
	\item\label{def:muc3} 
$\|T(t, s)P(s)\| \le N\left (\frac{\mu(t)}{\mu(s)}\right )^{-\nu}\mu(s)^\ep$ for $t\ge s$;
	\item\label{def:muc4} 
$\|T(t, s)(\Id-P(s))\|\le N\left (\frac{\mu(s)}{\mu(t)}\right )^{-\nu}\mu(s)^\ep$ for $t\le s$, where $T(t, s)$ denotes the inverse of the restriction of $T(s, t)$ to $\Ker P(t)$.
    \end{enumerate}

\end{definition}

Given an evolution family $\{T(t, s)\}_{t\ge s\ge 0}$ and two growth rates $\mu, \eta \colon [0, \infty)\to [1, \infty)$, set
\begin{equation}\label{newef}
T^{\mu, \eta}(t, s):=T(\mu^{-1}(\eta(t)), \mu^{-1}(\eta(s))), \quad t\ge s\ge 0.
\end{equation}
By arguing as in the proof of~\cite[Proposition 1]{DD} one can easily show that $\{T^{\mu, \eta}(t, s)\}_{t\ge s \ge 0}$ is an evolution family. Observe that 
\begin{equation}\label{inverse}
T(t, s)=T^{\mu, \eta}(\eta^{-1}(\mu(t)), \eta^{-1}(\mu(s))), \quad t\ge s\ge 0.
\end{equation}

The proof of the following result is similar to the proof of~\cite[Proposition 2]{DD}.
\begin{theorem}\label{thmc}
Let $\{T(t, s)\}_{t\ge s\ge 0}$ be an evolution family and $\mu, \eta \colon [0, \infty)\to [1, \infty)$ growth rates. The following statements are equivalent:
\begin{enumerate}
\item $\{T(t, s)\}_{t\ge s\ge 0}$ admits a nonuniform $\mu$-dichotomy with exponents $(\nu, \ep)$;
\item $\{T^{\mu, \eta}(t, s)\}_{t\ge s \ge 0}$ admits a nonuniform $\eta$-dichotomy with exponents $(\nu, \ep)$.
\end{enumerate}
\end{theorem}

\begin{proof}
$i.\implies ii.$ Suppose that  $\{T(t, s)\}_{t\ge s\ge 0}$ admits a nonuniform $\mu$-dichotomy with exponents $(\nu, \ep)$ relative to a family of projections $\{P(t)\}_{t\ge 0}$. Set
\[
\tilde P(t):=P(\mu^{-1}(\eta(t))), \quad t\ge 0.
\]
Then,
\[
\begin{split}
\tilde P(t)T^{\mu, \eta}(t, s) &=P(\mu^{-1}(\eta(t)))T(\mu^{-1}(\eta(t)), \mu^{-1}(\eta(s))) \\
&=T(\mu^{-1}(\eta(t)), \mu^{-1}(\eta(s)))P(\mu^{-1}(\eta(s)))\\
&=T^{\mu, \eta}(t, s)\tilde P(s),
\end{split}
\]
for $t\ge s\ge 0$. Moreover, for such $t$ and $s$, the restriction of $T^{\mu, \eta}(t, s)$ to $\Ker \tilde P(s)$ coincides with the restriction of $T(\mu^{-1}(\eta(t)), \mu^{-1}(\eta(s)))$ to $\Ker P(\mu^{-1}(\eta(s)))$ which is invertible.

Furthermore, for $t\ge s\ge 0$ we have 
\[
\begin{split}
\|T^{\mu, \eta}(t, s)\tilde P(s)\| &=\|T(\mu^{-1}(\eta(t)), \mu^{-1}(\eta(s)))P(\mu^{-1}(\eta(s)))\| \\
&\le N \left (\frac{\mu(\mu^{-1}(\eta(t)))}{\mu(\mu^{-1}(\eta(s)))}\right )^{-\nu}\mu(\mu^{-1}(\eta(s)))^\ep \\
&=N\left (\frac{\eta(t)}{\eta(s)}\right )^{-\nu}\eta(s)^\ep.
\end{split}
\]
Similarly, for $0\le t\le s$ we have 
\[
\begin{split}
\|T^{\mu, \eta}(t, s)(\Id-\tilde P(s))\| &=\|T(\mu^{-1}(\eta(t)), \mu^{-1}(\eta(s)))(\Id-P(\mu^{-1}(\eta(s))))\| \\
&\le N \left (\frac{\mu(\mu^{-1}(\eta(s)))}{\mu(\mu^{-1}(\eta(t)))}\right )^{-\nu}\mu(\mu^{-1}(\eta(s)))^\ep \\
&=N\left (\frac{\eta(s)}{\eta(t)}\right )^{-\nu}\eta(s)^\ep.
\end{split}
\]
We conclude that $ii.$ holds.
 
$ii. \implies i.$ Suppose that  $\{T^{\mu, \eta}(t, s)\}_{t\ge s\ge 0}$ admits a nonuniform $\eta$-dichotomy with exponents $(\nu, \ep)$ relative to a family of projections $\{\tilde P(t)\}_{t\ge 0}$. Set
\[
P(t):=\tilde P(\eta^{-1}(\mu(t))), \quad t\ge 0.
\]
Then (see~\eqref{inverse}) we have that for $t\ge s\ge 0$,
\[
\begin{split}
P(t)T(t, s) &=\tilde P(\eta^{-1}(\mu(t)))T^{\mu, \eta}(\eta^{-1}(\mu(t)), \eta^{-1}(\mu(s))) \\
&=T^{\mu, \eta}(\eta^{-1}(\mu(t)), \eta^{-1}(\mu(s)))\tilde P(\eta^{-1}(\mu(s))) \\
&=T(t, s)P(s).
\end{split}
\]
Moreover, the restriction of $T(t, s)$ to $\Ker P(s)$ which coincides with the restriction of $T^ {\mu, \eta}(\eta^{-1}(\mu(t)), \eta^{-1}(\mu(s)))$ to $\Ker \tilde P(\eta^{-1}(\mu(s)))$ 
 is invertible. In addition, for $t\ge s\ge 0$,
 \[
 \begin{split}
\|T(t, s)P(s)\| &=\|T^{\mu, \eta}(\eta^{-1}(\mu(t)), \eta^{-1}(\mu(s)))\tilde P(\eta^{-1}(\mu(s)))\| \\
&\le N\left (\frac{\eta (\eta^{-1}(\mu(t)))}{\eta (\eta^{-1}(\mu(s)))}\right )^{-\nu}\eta(\eta^{-1}(\mu(s)))^\ep \\
&=N\left (\frac{\mu(t)}{\mu(s)}\right )^{-\nu}\mu(s)^\ep.
 \end{split}
 \]
 Similarly, for $0\le t\le s$ we have 
 \[
 \|T(t, s)(\Id-P(s))\|\le N\left (\frac{\mu(s)}{\mu(t)}\right )^{-\nu}\mu(s)^\ep.
 \]
 We conclude that $i.$ holds and the proof of the theorem is completed.
\end{proof}

\begin{remark}\label{remc1}
So far, we have considered evolution families over $\mathbb R^+$. However, one can easily adapt the arguments to deal with evolution families  over $\mathbb R$. In this case, a (continuous-time) growth rate is an arbitrary continuous and strictly increasing map $\mu \colon \mathbb R \to (0,+\infty)$ such that $\mu(0)=1$, $\lim_{t\to \infty}\mu(t)=\infty$ and $\lim_{t\to -\infty}\mu(t)=0$. Given two such growth rates $\mu, \eta\colon \mathbb R\to (0,+\infty)$ and an evolution family $\{T(t, s)\}_{t\ge s }$, one can consider an evolution family $\{T^{\mu, \eta}(t, s)\}_{t\ge s}$ as in~\eqref{newef} (but for $t, s\in \mathbb R$, $t\ge s$).  One can now establish the version of Theorem~\ref{thmc} for evolution families on $\mathbb R$ whose proof would be identical to the proof of Theorem~\ref{thmc}.
\end{remark}

The equivalence established in Theorem~\ref{thmc} can be used to produce genuinely nonuniform dichotomies associated with nonexponential growth rates. The following example, obtained from a nonuniform exponential dichotomy through the time change $\tau=\ln(1+t)$, exhibits a nonuniform polynomial dichotomy that cannot be replaced by a uniform one.

\begin{example}\label{ex:nonuniform-polynomial-continuous}
Let $X=\mathbb R^2$, endowed with the maximum norm, and consider the growth rates
\[
\mu(t)=1+t
\qquad\text{and}\qquad
\eta(\tau)=e^\tau,
\]
and define $b\colon[0,+\infty)\to\mathbb R$ by
\begin{equation}\label{eq:continuous-polynomial-b}
b(\tau)=
\begin{cases}
-4^{j-1}\tau+4^{2j-1},	& 4^j\leq\tau\leq 4^j+1 \text{ for some }j\in\mathbb N\\[1mm]
-5\cdot 4^{j-1}+\tau-1, & 4^j+1\leq\tau\leq 5\cdot 4^{j-1}+1 \text{ for some }j\in\mathbb N\\[1mm]
0, &\text{otherwise}
\end{cases}.
\end{equation}
The blocks $[4^j,5\cdot4^{j-1}+1]$, $j\in\mathbb N$, are pairwise disjoint, and the expressions $-4^{j-1}\tau+4^{2j-1}$ and $-5\cdot 4^{j-1}+\tau-1$ in \eqref{eq:continuous-polynomial-b} agree at $\tau=4^j+1$. Consider the differential equation
\begin{equation}\label{eq:continuous-polynomial-system}
x'=A(t)x,
\qquad
A(t)=\frac{1}{1+t}
\begin{pmatrix}
-1+b'(\ln(1+t)) & 0\\
0 & 1-b'(\ln(1+t))
\end{pmatrix},
\end{equation}
for $t\geq0$, where the equality defining $A(t)$ is understood almost everywhere. The evolution family associated with equation~\eqref{eq:continuous-polynomial-system} is
\begin{equation}\label{eq:continuous-polynomial-evolution}
T(t,s)=
\begin{pmatrix}
\displaystyle
\frac{1+s}{1+t} \
e^{b(\ln(1+t))-b(\ln(1+s))}&0\\[4mm]
0&\displaystyle
\frac{1+t}{1+s} \
e^{-b(\ln(1+t))+b(\ln(1+s))}
\end{pmatrix},
 \ \ t\geq s\geq0.
\end{equation}
Let
\[
P=
\begin{pmatrix}
1&0\\
0&0
\end{pmatrix}.
\]
Clearly, $P$ is invariant and the evolution on $\ker P$ is invertible. Note that
\begin{equation}\label{eq:continuous-polynomial-bounds}
-\frac{\tau}{4}\leq b(\tau)\leq0
\quad\text{for every }\tau\geq0, \qquad b'(\tau)\leq1
\quad\text{for almost every }\tau\geq0.
\end{equation}
Indeed, on the interval $[4^j,4^j+1]$ the function decreases linearly from $0$ to $-4^{j-1}$, while on the interval $[4^j+1,5\cdot4^{j-1}+1]$ it increases linearly from $-4^{j-1}$ to $0$. The lower bound in \eqref{eq:continuous-polynomial-bounds} follows immediately on both subintervals.

If $t\geq s$, then $b(\ln(1+t))\leq0$ and
$b(\ln(s+1))\geq-\frac14\ln(s+1)$. Consequently,
$$\|T(t,s)P\| =\frac{1+s}{1+t}
e^{b(\ln(1+t))-b(\ln(s+1))} \leq
\left(\frac{1+t}{1+s}\right)^{-1}(1+s)^{1/4} =\left(\frac{\mu(t)}{\mu(s)}\right)^{-1}\mu(s)^{1/4}.
$$
For $0\leq t\leq s$, using the inverse of the evolution on $\ker P$, we similarly obtain
\[
\begin{split}
\|T(t,s)(\Id-P)\|
&= \frac{1+t}{1+s} \
e^{-b(\ln(1+t))+b(\ln(s+1))} \leq
\left(\frac{1+s}{1+t}\right)^{-1}(1+t)^{1/4}\\
&\leq \left(\frac{\mu(s)}{\mu(t)}\right)^{-1}\mu(s)^{1/4}.
\end{split}
\]
Hence, equation~\eqref{eq:continuous-polynomial-system} admits a nonuniform $\mu$-dichotomy with $N=1$, $\nu=1$, and $\varepsilon=1/4$.
	
We next describe the time rescaling explicitly. Since $\mu^{-1}(e^\tau)=e^\tau-1$, it follows from \eqref{eq:continuous-polynomial-evolution} that
\begin{equation}\label{eq:continuous-polynomial-rescaled}
\begin{split}
T^{\mu,\eta}(\tau,\sigma)
&=T(e^\tau-1,e^\sigma-1)\\
&=
\begin{pmatrix}
e^{-(\tau-\sigma)+b(\tau)-b(\sigma)} & 0\\
0 & e^{\tau-\sigma-b(\tau)+b(\sigma)}
\end{pmatrix},
\qquad \tau\geq\sigma\geq0.
\end{split}
\end{equation}
The estimates already established show that this family admits a nonuniform exponential dichotomy with exponents
$$\nu=1 \qquad\text{and}\qquad \varepsilon=\frac14.$$
	
We will now show that the nonuniformity cannot be removed. Set
$$ s_j=e^{4^j+1}-1 \qquad\text{and}\qquad t_j=e^{5\cdot 4^{j-1}+1}-1. $$
For $t\in[s_j,t_j]$, we have
\[
\ln(1+t)\in[4^j+1,5\cdot 4^{j-1}+1]
\]
and thus $b'(\ln(1+t))=1$ almost everywhere. Hence, $A(t)=0$ almost everywhere on $[s_j,t_j]$ and
\begin{equation}\label{eq:continuous-polynomial-identity}
T(t_j,s_j)=\Id.
\end{equation}
Moreover,
\begin{equation}\label{eq:continuous-polynomial-ratio}
\frac{\mu(t_j)}{\mu(s_j)}=\frac{1+t_j}{1+s_j}=e^{4^{j-1}}\longrightarrow+\infty.
\end{equation}
	
Assume, for contradiction, that a uniform dichotomy with respect to $\mu$ exists. Let $(R(t))_{t\geq0}$ be the associated invariant family of projections, and let $D\geq1$ and $\alpha>0$ be the dichotomy constants. Invariance and \eqref{eq:continuous-polynomial-identity} imply that $R(t_j)=R(s_j)$. Applying the stable and unstable estimates to the endpoints of $[s_j,t_j]$ and using \eqref{eq:continuous-polynomial-ratio}, we obtain
\[
\|R(s_j)\|\leq De^{-\alpha 4^{j-1}}
\qquad\text{and}\qquad
\|\Id-R(t_j)\|\leq De^{-\alpha 4^{j-1}}.
\]
Therefore,
\[
1=\|\Id\|
\leq\|R(s_j)\|+\|\Id-R(s_j)\|
\leq2De^{-\alpha 4^{j-1}},
\]
which is impossible as $j\to\infty$. Thus, no uniform $\mu$-dichotomy exists.
	
This example provides a concrete
illustration of Theorem~\ref{thmc}: the nonuniform polynomial dichotomy of
the original system becomes a nonuniform exponential dichotomy after the
time rescaling $\tau=\ln(1+t)$. The example also shows that the nonuniformity
cannot, in general, be removed.
\end{example}

\subsection{Robustness of general dichotomies}
Throughout this subsection, we will fix an evolution family $\{T(t, s)\}_{t\ge s\ge 0}\subset \mathcal B(X)$ and a continuous map $B\colon [0, \infty)\to \mathcal B(X)$. Furthermore, we will assume that there is an evolution family $\{T_B(t, s)\}_{t\ge s\ge 0}\subset \mathcal B(X)$ such that 
\begin{equation}\label{348}
T_B(t, s)=T(t, s)+\int_s^t T(t, r)B(r)T_B(r, s)\, dr\quad \text{for $t\ge s\ge 0$.}
\end{equation}
\begin{theorem}\label{robustness}
Let $\mu \colon [0, \infty)\to [1, \infty)$ be a continuously differentiable growth rate such that $\mu'(t)>0$ for each $t\in [0, \infty)$.
Suppose the following:
\begin{enumerate}
\item $\{T(t, s)\}_{t\ge s\ge 0}$ admits a nonuniform $\mu$-dichotomy with exponents $(\nu, \ep)$;
\item there is $c>0$ such that 
\begin{equation}\label{smallpert}
\|B(t)\| \le \frac{c\mu'(t)}{\mu(t)^{1+\ep}}, \quad t\ge 0.
\end{equation}
\end{enumerate}
Then, provided that $c$ is sufficiently small, $\{T_B(t, s)\}_{t\ge s\ge 0}$ admits a nonuniform $\mu$-dichotomy.
\end{theorem}

\begin{proof}
Let $\eta(t)=e^t$, $t\ge 0$.
By~\eqref{348} we have that
\begin{equation}\label{pertubation}
\begin{split}
&T_B^{\mu, \eta}(t, s)\\
&=T_B(\mu^{-1}(e^t), \mu^{-1}(e^s))\\
&=T(\mu^{-1}(e^t), \mu^{-1}(e^s))
+\int_{\mu^{-1}(e^s)}^{\mu^{-1}(e^t)}T(\mu^{-1}(e^t), r)B(r)T_B(r, \mu^{-1}(e^s))\, dr \\
&=T^{\mu, \eta}(t, s)+\int_s^t T(\mu^{-1}(e^t), \mu^{-1}(e^z))B(\mu^{-1}(e^z))T_B(\mu^{-1}(e^z), \mu^{-1}(e^s)) \frac{e^z}{\mu'(\mu^{-1}(e^z))}\, dz \\
&=T^{\mu, \eta}(t, s)+\int_s^t T^{\mu, \eta}(t, z) B(\mu^{-1}(e^z))T_B^{\mu, \eta}(z, s)\frac{e^z}{\mu'(\mu^{-1}(e^z))} \, dz \\
&=T^{\mu, \eta}(t, s)+\int_s^t T^{\mu, \eta}(t, z) \tilde B(z) T_B^{\mu, \eta}(z, s)\,  dz, \\
\end{split}
\end{equation}
for $t\ge s\ge 0$,
where
\[
\tilde B(z):=\frac{e^z}{\mu'(\mu^{-1}(e^z))} B(\mu^{-1}(e^z)), \quad z\ge 0.
\]
Note that~\eqref{smallpert} gives that
\begin{equation}\label{smallB2}
\|\tilde B(z)\|  \le ce^{-\ep z}, \quad z\ge 0.
\end{equation}
By Theorem~\ref{thmc}, the evolution family $\{T^{\mu, \eta}(t, s)\}_{t\ge s\ge 0}$ admits a nonuniform exponential dichotomy with exponents $(\nu, \ep$). Now, it follows from~\eqref{pertubation}, \eqref{smallB2} and~\cite[Theorem 3]{BDV} (applied to the case when $I=\{0\}$) that the evolution family $\{T_B^{\mu, \eta}(t, s)\}_{t\ge s\ge 0}$ admits a nonuniform exponential dichotomy, provided that $c$ is sufficiently small. Hence, it follows from Theorem~\ref{thmc} that $\{T_B(t, s)\}_{t\ge s\ge 0}$ admits  a nonuniform $\mu$-dichotomy, which completes the proof of the theorem. 
\end{proof}

\begin{remark}
One can easily establish the version of Theorem~\ref{robustness} for evolution families over $\mathbb R$ (see the discussion in Remark~\ref{remc1}). A result of this type has been proved in~\cite[Theorem 1]{BV}. In order to compare the two results, we write $\mu$ in the form  $\mu(t)=e^{\rho(t)}$.  Now~\eqref{smallpert} reduces to 
\begin{equation}\label{824pert}
\|B(t)\| \le \frac{c\rho'(t)}{e^{\varepsilon \rho(t)}}, \quad t\ge 0.
\end{equation}
In comparison, \cite[Theorem 1]{BV} (adjusted to the case of evolution families on $\mathbb R^+$) requires that 
\begin{equation}\label{838pert}
\|B(t)\| \le \frac{c\rho'(t)}{e^{3\varepsilon \rho(t)}}, \quad t\ge 0,
\end{equation}
and in addition that $\{T(t, s)\}$ admits a nonuniform $\mu$-dichotomy with exponents $(\nu, \ep)$ that satisfy $2\ep<\nu$. 

Clearly, \eqref{824pert} is weaker than~\eqref{838pert}. In addition, in Theorem~\ref{robustness} we make no assumptions regarding the constants associated with the dichotomy of $\{T(t, s)\}$. Finally, our proof of Theorem~\ref{robustness} is completely different from  the proof of~\cite[Theorem 1]{BV}, which makes no use of the relationship between exponential dichotomies and dichotomies with growth rates.
\end{remark}

\subsection{Sacker-Sell spectrum}
Throughout this subsection we will consider an evolution family $\{T(t, s)\}_{t\ge s\ge 0}$ associated with an ordinary differential equation
\[
x'=A(t)x, \quad t\ge 0,
\]
where $A\colon [0, \infty)\to \mathcal M_d$ is a locally integrable map and $\mathcal M_d$ denotes the space of all real matrices of order $d$.

In addition, we fix a continuously differentiable growth rate $\mu \colon [0, \infty)\to [1, \infty)$ such that $\mu'(t)>0$ for each $t$. For $\lambda \in \mathbb R$ we consider the equation
\begin{equation}\label{947}
x'=\left (A(t)-\lambda \frac{\mu'(t)}{\mu(t)}\Id \right )x, \quad t\ge 0.
\end{equation}
Note that the evolution family associated to~\eqref{947} is given by 
\[
T_\lambda (t, s)=\left (\frac{\mu(t)}{\mu(s)}\right )^{-\lambda}T(t, s).
\]
By $\Sigma_{A, \mu}^{ND}$ we denote the set of all $\lambda \in \mathbb R$ such that $\{T_\lambda (t, s)\}_{t\ge s \ge 0}$ does not admit a nonuniform $\mu$-dichotomy with exponents $(\nu, \ep)$ with $\ep<\nu$.

Finally, for $\lambda \in \mathbb R$ we also consider the equation
\begin{equation}\label{956}
x'=(\tilde A(t)-\lambda \Id)x, \quad t\ge 0,
\end{equation}
where \[\tilde A(t):=\frac{e^t}{\mu'(\mu^{-1}(e^t))}A(\mu^{-1}(e^t)), \quad t\ge 0.\] We note that the evolution family corresponding to~\eqref{956} is given by 
\[
\tilde T_\lambda (t, s)=e^{-\lambda (t-s)}T(\mu^{-1}(e^t), \mu^{-1}(e^s)).
\]
Let $\Sigma_{\tilde A, exp}^{ND}$ denote the set of all $\lambda \in \mathbb R$ such that $\{\tilde T_\lambda (t,s)\}_{t\ge s\ge 0}$ does not admit a nonuniform exponential dichotomy with exponents $(\nu, \ep)$ with $\ep<\nu$.

The following result is a direct consequence of Theorem~\ref{thmc} and a simple observation that 
\[
T_\lambda(\mu^{-1}(e^t), \mu^{-1}(e^s))=\tilde T_\lambda(t, s).
\]
\begin{corollary}
    We have that $\Sigma_{A, \mu}^{ND}=\Sigma_{\tilde A, exp}^{ND}$.
\end{corollary}

\begin{remark}
The structure of $\Sigma_{A, \mu}^{ND}$ was studied in~\cite{BDV2, Silva}. We have shown here that this structure can be deduced directly from the structure of the ``classical'' dichotomy spectrum introduced with respect to the notion of a (nonuniform) exponential dichotomy. 
\end{remark}

\subsection{A Sternberg-type linearization theorem}
Throughout this subsection, we will consider a locally integrable $A\colon \mathbb R\to \mathcal M_d$ and the associated equation
\begin{equation}\label{209lde}
x'=A(t)x, \quad t\in \mathbb R.
\end{equation}
Let $\mu \colon \mathbb R\to (0, \infty)$ be a continuously differentiable growth rate such that $\mu'(t)>0$ for each $t\in \mathbb R$ (see Remark~\ref{remc1}). By $\Sigma_{A, \mu}^{UD}$ we will denote the set of all $\lambda \in \mathbb R$ such that  the equation
\[
x'=\left (A(t)-\lambda \frac{\mu'(t)}{\mu(t)}\Id \right )x, \quad t\in \mathbb R
\]
does not admit a uniform $\mu$-dichotomy, i.e. nonuniform $\mu$-dichotomy with exponents $(\nu, 0)$ for some $\nu>0$.
By arguing as in the previous subsection, we have that $\Sigma_{A, \mu}^{UD}=\Sigma_{\tilde A, exp}^{UD}$, 
where \begin{equation}\label{AA}\tilde A(t):=\frac{e^t}{\mu'(\mu^{-1}(e^t))}A(\mu^{-1}(e^t)), \quad t\in \mathbb R.\end{equation}

In the sequel, we will assume that there are $K, a>0$ such that 
\begin{equation}\label{12:22}
\|T(t, s)\| \le K\left (\frac{\mu(t)}{\mu(s)}\right )^a \quad \text{for $t\ge s$,}
\end{equation}
and
\[
\|T(t, s)\| \le K\left (\frac{\mu(s)}{\mu(t)}\right )^a \quad \text{for $t\le s$,}
\]
where $\{T(t, s)\}_{t, s\in \mathbb R}$ is an evolution family associated with~\eqref{209lde}. Let $T^{\mu, exp}(t, s)$ be defined as in~\eqref{newef} with $\eta(t)=e^t$, and for $t, s\in \mathbb R$. The above requirements imply that 
\[
\|T^{\mu, exp}(t, s)\|\le Ke^{a|t-s|}, \quad t, s\in \mathbb R.
\]
Consequently, by~\cite[Theorem 3.1]{Siegmund}
\[
\Sigma:=\Sigma_{A, \mu}^{UD}=\Sigma_{\tilde A, exp}^{UD}=[a_1, b_1]\cup \ldots \cup [a_n, b_n]
\]
for some $1\le n\le d$ and 
\[
a_1\le b_1<\ldots <a_n \le b_n.
\]
In the sequel, we will also consider nonlinear perturbations of~\eqref{209lde} of the form
\begin{equation}\label{326se}
x'=A(t)x+f(t, x), \quad t\in \mathbb R,
\end{equation}
where $f\colon \mathbb R\times \mathbb R^d\to \mathbb R^d$.
By $\{U(t, s)\}_{t, s\in \mathbb R}$ we will denote a (nonlinear) evolution family associated to~\eqref{326se} (we assume that the solutions of~\eqref{326se} are globally defined).

For $k\in \mathbb N$, let $\mathcal O_\mu^k$ denote the set of all $f\colon \mathbb R\times \mathbb R^d \to \mathbb R^d$ which are $C^k$-Carath\' eodory (as in~\cite[p.1280]{CDS}) such that:
\begin{itemize}
\item for $t\in \mathbb R$, $f(t, 0)=0$ and $D_x f(t, 0)=0$;
\item there is $M>0$ such that 
\begin{equation}\label{546db}
\|D_x^jf(t, x)\| \le M\frac{\mu'(t)}{\mu(t)} \quad \text{for a.e. $t\in \mathbb R$, $x\in \mathbb R^d$ and $0\le j \le k$,}
\end{equation}
where $D_x^j f$ denotes the $j$-th derivative of $f$ with respect to the variable $x$.
\end{itemize}
In the particular case when $\mu(t)=e^t$, we will write $\mathcal O^k$ instead of $\mathcal O_\mu^k$.

The following is a version of~\cite[Theorem 5]{CDS} for dichotomies with growth rates.
\begin{theorem}\label{Stern}
Let $\ell \in \mathbb N$. Then there exists $k\in \mathbb N$
with the following property: if $0\notin \Sigma$, $f\in \mathcal O_\mu^{k+2}$ and 
\[
[a_i, b_i]\cap [\langle a, q\rangle, \langle b, q\rangle]=\emptyset \quad \text{for $1\le i \le n$ and $q\in \mathbb N_0^n$ with $2\le |q| \le k$,}
\]
where $a=(a_1, \ldots, a_n)$, $b=(b_1, \ldots, b_n)$, $|q|=\sum_{i=1}^n q_i$, and $\langle \cdot, \cdot \rangle$ denotes the standard scalar product on $\mathbb R^n$, then~\eqref{209lde} and~\eqref{326se} are $C^\ell $ equivalent. Namely, there are $p, \tilde p>0$, $r\in (0, p)$, $\tilde r\in (0, \tilde p)$ and continuous maps $H_1\colon \mathbb R\times B_r(0)\to \mathbb R^d$, $H_2\colon \mathbb R\times B_{\tilde r}(0)\to \mathbb R^d$ such that:
\begin{itemize}
\item $H_i(t, 0)=0$ for $t\in \mathbb R$, $i=1, 2$. Moreover, $H_1(t, \cdot)(B_r(0))\subset B_{\tilde p}(0)$ and $H_2(t, \cdot)(B_{\tilde r}(0))\subset B_p(0)$;
\item for each $t\in \mathbb R$ and $i\in \{1, 2\}$, $H_i(t, \cdot)$ is of class $C^\ell$;
\item $H_1(t, H_2(t, x))=x$ for each $t\in \mathbb R$  and $x\in B_{\tilde r}(0)$ with $H_2(t, x)\in B_r(0)$. Analogously, $H_2(t, H_1(t, x))=x$ for $t\in \mathbb R$ and $x\in B_r(0)$ with $H_1(t, x)\in B_{\tilde r}(0)$;
\item $H_1(t, T(t, s)x)=U(t, s)H_1(s, x)$ for $t, s\in \mathbb R$ and $x\in B_r(0)$ with $T(t, s)x\in B_r(0)$. Analogously, $T(t, s)H_2(s, x)=H_2(t, U(t, s)x)$ for $t, s\in \mathbb R$ and  $x\in B_{\tilde r}(0)$ with $U(t, s)x\in B_{\tilde r}(0)$;
\item $\lim_{x\to 0}H_1(t, x)=\lim_{x\to 0}H_2(t, x)=0$ uniformly in $t$.
\end{itemize}
\end{theorem}

\begin{proof}
Let $k$ be given by~\cite[Theorem 5]{CDS}.
We define $g\colon \mathbb R\times \mathbb R^d \to \mathbb R^d$ by 
\[
g(t, x):=\frac{e^t}{\mu'(\mu^{-1}(e^t))}f(\mu^{-1}(e^t), x), \quad (t, x)\in \mathbb R\times \mathbb R^d.
\]
Clearly, $g(t, 0)=0$ and $D_xg(t, 0)=0$ for each $t\in \mathbb R$. Moreover, \eqref{546db} (with $k+2$ instead of $k$) implies that 
\[
\|D_x^j g(t, x)\|  =\frac{e^t}{\mu'(\mu^{-1}(e^t))}\|D_x^j f(\mu^{-1}(e^t), x)\|\le M\frac{e^t}{\mu'(\mu^{-1}(e^t))}\frac{\mu'(\mu^{-1}(e^t))}{\mu(\mu^{-1}(e^t))}=M,
\]
for a.e. $t\in \mathbb R$, $x\in \mathbb R^d$ and $0\le j\le k+2$. Therefore, $g\in \mathcal O^{k+2}$. It follows from~\cite[Theorem 5]{CDS} that 
\begin{equation}\label{8:13}
x'=\tilde A(t)x, \quad t\in \mathbb R,
\end{equation}
and 
\begin{equation}\label{8:14}
x'=\tilde A(t)x(t)+g(t, x), \quad t\in \mathbb R
\end{equation}
are $C^\ell$ equivalent. Recalling that the evolution  family associated to~\eqref{8:13} is given by $\{T^{\mu, exp}(t, s)\}_{t, s\in \mathbb R}$, and that the evolution family corresponding to~\eqref{8:14} is given by $\{U(\mu^{-1}(e^t), \mu^{-1}(e^s))\}_{t, s\in \mathbb R}$, this means there are $p, \tilde p>0$, $r\in (0, p)$, $\tilde r\in (0, \tilde p)$ and continuous maps $H\colon \mathbb R\times B_r(0)\to \mathbb R^d$, $\tilde H\colon \mathbb R\times B_{\tilde r}(0)\to \mathbb R^d$ such that:
\begin{itemize}
\item $H(t, 0)=0$ and $\tilde H(t, 0)=0$ for $t\in \mathbb R$. Moreover, $H(t, \cdot)(B_r(0))\subset B_{\tilde p}(0)$ and $\tilde H(t, \cdot)(B_{\tilde r}(0))\subset B_p(0)$;
\item for each $t\in \mathbb R$, $H(t, \cdot)$ and $\tilde H(t, \cdot)$ are of class $C^\ell$;
\item $H(t, \tilde H(t, x))=x$ for each $t\in \mathbb R$  and $x\in B_{\tilde r}(0)$ with $\tilde H(t, x)\in B_r(0)$. Analogously, $\tilde H(t, H(t, x))=x$ for $t\in \mathbb R$ and $x\in B_r(0)$ with $H(t, x)\in B_{\tilde r}(0)$;
\item \[H(t, T(\mu^{-1}(e^t), \mu^{-1}(e^s))x)=U(\mu^{-1}(e^t), \mu^{-1}(e^s))H(s, x)\] for $t, s\in \mathbb R$ and $x\in B_r(0)$ with $T^{\mu, exp}(t, s)x\in B_r(0)$. Analogously, \[T(\mu^{-1}(e^t), \mu^{-1}(e^s))\tilde H(s, x)=\tilde H(t, U(\mu^{-1}(e^t), \mu^{-1}(e^s))x) \] for $t, s\in \mathbb R$ and  $x\in B_{\tilde r}(0)$ with $U(\mu^{-1}(e^t), \mu^{-1}(e^s))x\in B_{\tilde r}(0)$;
\item $\lim_{x\to 0}H(t, x)=\lim_{x\to 0}\tilde H(t, x)=0$ uniformly in $t$.
\end{itemize}
For $t\in \mathbb R$ and $x\in B_r(0)$, set $H_1(t, x)=H(\ln \mu(t), x)$. Analogously, for $t\in \mathbb R$ and $x\in B_{\tilde r}(0)$, let $H_2(t,x):=\tilde H(\ln \mu(t), x)$. It is straightforward to verify that $H_i$, $i=1, 2$ satisfy the desired properties. 
\end{proof}

\subsection{Applications to admissibility}
Throughout this subsection we deal with an evolution family $\{T(t, s)\}_{t\ge s}$ over  $\mathbb R$ acting on a Banach space $X=(X, \|\cdot \|)$,  and a continuously differentiable growth rate $\mu \colon \mathbb R\to (0, \infty)$ such that $\mu'(t)>0$ for each $t$.

Let $Y_\infty$ denote the space of continuous maps $x\colon \mathbb R\to X$ such that 
\[
\|x\|_\infty:=\sup_{t\in \mathbb R}\|x(t)\|<\infty.
\]
Then $(Y_\infty, \|\cdot \|_\infty)$ is a Banach space.

\begin{theorem}\label{TM-ADMISS}
Suppose that there exists $K, a>0$ such that~\eqref{12:22} holds. Then
the following properties are equivalent:
\begin{enumerate}
\item[(i)] $\{T(t, s)\}_{t\ge s}$ admits a uniform $\mu$-dichotomy;
\item[(ii)] for each $y\in Y_\infty$, there is a unique $x\in Y_\infty$ such that 
\begin{equation}\label{adm}
x(t)=T(t, s)x(s)+\int_s^t \frac{\mu'(r)}{\mu(r)}T(t, r)y(r)\, dr \quad \text{for $t\ge s$.}
\end{equation}
\end{enumerate}
\end{theorem}

\begin{proof}
$(ii)\implies (i)$  We consider the evolution family $\{T^{\mu, exp}(t, s)\}_{t\ge s}$ as in the proof of Theorem~\ref{Stern}.  Note that 
\[
\|T^{\mu, exp}(t, s)\| \le Ke^{a(t-s)}, \quad t\ge s.
\]
We claim that $\{T^{\mu, exp}(t,s)\}_{t\ge s }$ admits a uniform exponential dichotomy. To this end, we take an arbitrary $y\in Y_\infty$, and define $\bar y\colon \mathbb R\to X$ by 
\[
\bar y(t):=y(\ln \mu(t)), \quad t\in \mathbb R.
\]
Clearly, $\bar y\in Y_\infty$. Due to our assumption, there exists a unique $\bar x\in Y_\infty$ such that 
\begin{equation}\label{950adm}
\bar x(t)=T(t, s)\bar x(s)+\int_s^t \frac{\mu'(r)}{\mu(r)}T(t, r)\bar y(r)\, dr, \quad t\ge s.
\end{equation}
We define $x\colon \mathbb R\to X$ by 
\[
x(t):=\bar x(\mu^{-1}(e^t)), \quad t\in \mathbb R.
\]
Then $x\in Y_\infty$. Moreover, \eqref{950adm} and the substitution $r=\mu^{-1}(e^z)$ gives that
\[
\begin{split}
x(t)=\bar x(\mu^{-1}(e^t)) &=T(\mu^{-1}(e^t), \mu^{-1}(e^s))\bar x(\mu^{-1}(e^s))\\
&\phantom{=}+\int_{\mu^{-1}(e^s)}^{\mu^{-1}(e^t)}\frac{\mu'(r)}{\mu(r)}T(\mu^{-1}(e^t), r)y(\ln \mu(r))\, dr \\
&=T^{\mu, exp}(t, s)x(s)+\int_s^t T^{\mu, \exp}(t, z)y(z)\, dz
\end{split}
\]
for $t\ge s$. We have proved that for each $y\in Y_\infty$, there exists $x\in Y_\infty$ such that 
\begin{equation}\label{1020adm}
x(t)=T^{\mu, \exp}(t, s)x(s)+\int_s^t T^{\mu, \exp}(t, z)y(z)\, dz, \quad \text{for $t\ge s$.}
\end{equation}
Moreover, such $x$ is unique. Indeed, suppose that $x_i$, $i=1, 2$ are such that 
\[
x_i(t)=T^{\mu, \exp}(t, s)x_i(s)+\int_s^t T^{\mu, \exp}(t, z)y(z)\, dz \quad \text{for $t\ge s$ and $i=1, 2$.}
\]
Then, $\tilde x:=x_1-x_2\in Y_\infty$ and 
\[
\tilde x(t)=T^{\mu, exp}(t, s)\tilde x(s), \quad t\ge s.
\]
Let 
\[
\underline{x}(t):=\tilde x(\ln \mu(t)), \quad t\in \mathbb R.
\]
Then, 
\[
\underline{x}(t)=T(t, s)\underline{x}(s), \quad t\ge s, 
\]
which means that 
\[
\underline{x}(t)=T(t, s)\underline{x}(s)+\int_s^t \frac{\mu'(r)}{\mu(r)}T(t, r)w(r)\, dr \quad \text{for $t\ge s$,}
\]
where $w\in Y_\infty$ is given by $w(t)=0$ for $t\in \mathbb R$. Due to the uniqueness in $(ii)$, we conclude that $\underline{x}=0$, which means that $x_1=x_2$. 

It now follows from~\cite[Theorem 4.28]{CL} that $\{T^{\mu, exp}(t, s)\}_{t\ge s}$ admits a uniform exponential dichotomy, which implies $(i)$.

$(i)\implies (ii)$ It follows by an analogous reasoning. Since $\{T(t, s)\}_{t\ge s}$ admits a uniform $\mu$-dichotomy, $\{T^{\mu, exp}(t, s)\}_{t\ge s}$ admits a uniform exponential dichotomy. By~\cite[Theorem 4.28]{CL}, for each $y\in Y_\infty$, there exists a unique $x\in Y_\infty$ such that~\eqref{1020adm}. By reversing the arguments in the proof of the implication $(ii)\implies (i)$, one can show that $(ii)$ holds.
\end{proof}

\begin{remark}
Although to the best of our knowledge Theorem~\ref{TM-ADMISS} has not been formulated earlier, a similar result dealing with evolution families over $\mathbb R^+$ follows from the main results in~\cite{DLP}. However, our arguments are again simpler than those in~\cite{DLP}. 
\end{remark}

\subsection{Lyapunov-type characterization of general dichotomies}
Throughout this subsection, $X$ will be an arbitrary Hilbert space and $\mu \colon \mathbb R\to (0, \infty)$ is a continuously differentiable growth rate such that $\mu'(t)>0$ for each $t$. In addition, we consider~\eqref{209lde} with $A\colon \mathbb R\to \mathcal B(X)$  continuous.

In the particular case when $\mu(t)=e^t$, the following result coincides with~\cite[Corollary 4.48]{CL}.
\begin{theorem}\label{Lyapunov}
Suppose that there exists $M>0$ such that 
\begin{equation}\label{1010b}
\|A(t)\| \le M \frac{\mu'(t)}{\mu(t)}, \quad t\in \mathbb R.
\end{equation}
The following properties are equivalent:
\begin{enumerate}
\item[(i)] \eqref{209lde} admits a uniform $\mu$-dichotomy;
\item[(ii)] there exists a bounded and differentiable map $W\colon \mathbb R\to \mathcal B(X)$ such that $W(t)$ is invertible and self-adjoint for each $t$, $\sup_{t\in \mathbb R}\|W(t)^{-1}\|<\infty$, and for some $\delta>0$ we have 
\begin{equation}\label{lyap}
\left \langle W'(t)x+A(t)^*W(t)x+W(t)A(t)x, x \right \rangle \le -\delta\frac{\mu'(t)}{\mu(t)}\langle x, x\rangle,
\end{equation}
for every $t\in \mathbb R$ and $x\in X$.
\end{enumerate}
\end{theorem}

\begin{proof}
Let $\tilde A\colon \mathbb R\to \mathcal B(X)$ be given by~\eqref{AA}. Note that~\eqref{1010b} implies that $\sup_{t\in \mathbb R}\|\tilde A(t)\|<\infty$.

$(i)\implies (ii)$ Assume that~\eqref{209lde} admits a uniform $\mu$-dichotomy. Then the  linear equation
\begin{equation}\label{1:14}
x'=\tilde A(t)x, \quad t\in \mathbb R
\end{equation}
admits a uniform exponential dichotomy.  By~\cite[Corollary 4.48]{CL} there is a bounded and differentiable map $\tilde W\colon \mathbb R\to \mathcal B(X)$ such that $\tilde W(t)$ is self-adjoint and invertible for each $t$, $\sup_{t\in \mathbb R}\|\tilde W(t)^{-1}\|<\infty$, and for some $\delta>0$ we have 
\begin{equation}\label{1104}
    \left \langle \tilde W'(t)x+\tilde A(t)^*\tilde W(t)x+\tilde W(t)\tilde A(t)x, x\right \rangle \le -\delta \langle x, x \rangle,
\end{equation}
for every $t\in \mathbb R$ and $x\in X$. Set $W(t):=\tilde W(\ln \mu(t))$, $t\in \mathbb R$. Then $W$ is bounded and  differentiable, $W(t)$ is self-adjoint and invertible for each $t$ and $\sup_{t\in \mathbb R}\|W(t)^{-1}\|<\infty$. Moreover, \eqref{1104} implies that 
\[
\left \langle \tilde W'(\ln \mu(t))x+\frac{\mu(t)}{\mu'(t)}A(t)^*W(t)x+\frac{\mu(t)}{\mu'(t)}W(t)A(t)x, x\right \rangle \le -\delta \langle x, x\rangle,
\]
for each $t\in \mathbb R$ and $x\in X$. By taking into account that $W'(t)=\frac{\mu'(t)}{\mu(t)}\tilde W'(\ln \mu(t))$, we conclude that~\eqref{lyap} holds.

$(ii)\implies (i)$ It follows by reversing the arguments in the proof that $(i)$ implies $(ii)$. More precisely, suppose that there exists $W\colon \mathbb R\to \mathcal B(X)$ as in $(ii)$. Set \[\tilde W(t):=W(\mu^{-1}(e^t)), \quad  t\in \mathbb R.\]
Then $\tilde W\colon \mathbb R\to \mathcal B(X)$ is bounded  and differentiable, $\tilde W(t)$ is self-adjoint and invertible for each $t\in \mathbb R$, $\sup_{t\in \mathbb R}\|\tilde W(t)^{-1}\|<\infty$ and~\eqref{1104} holds. Then~\eqref{1:14} admits a uniform exponential dichotomy, which implies $(i)$.
\end{proof}

\begin{remark}
To the best of our knowledge, Theorem~\ref{Lyapunov} is a completely new result as the previous results devoted to the characterization of general dichotomies in terms of quadratic Lyapunov functions where restricted to the finite-dimensional setting (see~\cite{BCV}).
\end{remark}

\subsection{Stable manifolds for semilinear integral equations}
We consider an evolution family $\{T(t, s)\}_{t\ge s\ge 0}$ on an arbitrary Banach space $X$. Let $f\colon \mathbb R^+\times X\to X$ be a continuous map and consider an integral equation
\begin{equation}\label{integeq}
u(t)=T(t, s)u(s)+\int_s^t T(t, r)f(r, u(r))\, dr,  \quad t\ge s.
\end{equation}

The following definition is taken from~\cite[Definition 4.2]{Huy}.
\begin{definition}
A set $\mathbf S\subset \mathbb R^+\times X$ is said to be an \emph{invariant stable manifold} for~\eqref{integeq} if for every $t\ge 0$, there is a splitting $X=X_0(t)\oplus X_1(t)$ of $X$ into a direct sum of two closed subspaces $X_i(t)$, $i=0, 1$ such that 
\[
\inf_{t\ge 0}\inf \{\|v_0+v_1\|: \ v_i\in X_i(t) \ \text{and} \ \|v_i\|=1 \ \text{for $i=0, 1$}\}>0
\]
and if there exists a family of Lipschitz continuous maps $g_t\colon X_0(t)\to X_1(t)$, $t\ge 0$  with  Lipschitz constants independent of $t$, such that:
\begin{enumerate}
\item[(i)] $\mathbf S=\{(t, v+g_t(v)): \ t\ge 0, \ v\in X_0(t)\}$;
\item[(ii)] for each $t\ge 0$, $\mathbf S_t$ is homeomorphic to $X_0(t)$, where
\[
\mathbf S_t:=\{w\in X: \ (t, w)\in \mathbf  S\}=\{v+g_t(v): \ v\in X_0(t)\};
\]
\item[(iii)] to each $v\in \mathbf S_{t_0}$ there corresponds a unique solution $u$ of~\eqref{integeq} on $[t_0, \infty)$ such that $u(t_0)=v$ and $\sup_{t\ge t_0}\|u(t)\|<\infty$;
\item[(iv)] if $u$ is a solution of~\eqref{integeq} satisfying $u(t_0)\in \mathbf S_{t_0}$ and $\sup_{t\ge t_0}\|u(t)\|<\infty$, then $u(t)\in \mathbf S_t$ for each $t\ge t_0$.
\end{enumerate}
\end{definition}

We need the following result, which follows from~\cite[Theorem 4.6]{Huy}.
\begin{theorem}\label{HuyTM}
Suppose that $f$ satisfies the following:
\begin{enumerate}
\item[(a)] $f(t, 0)=0$ for $t\ge 0$;
\item[(b)] there is $c>0$ such that
\begin{equation}\label{lip}
\|f(t, x)-f(t, y)\| \le c\|x-y\|, \quad \text{for $t\ge 0$ and $x, y\in X$.}
\end{equation}
\end{enumerate}
Moreover, suppose that $\{T(t, s)\}_{t\ge s\ge 0}$ admits a uniform exponential dichotomy. Provided that $c$ is sufficiently small, there is an invariant stable manifold $\mathbf S$ for~\eqref{integeq}. Moreover, there are $\nu, C_\nu>0$ such that for each $t_0>0$ and solutions $u_1, u_2$ of~\eqref{integeq} on $[t_0, \infty)$ such that $u_i(t_0)=v_i+g_{t_0}(v_i)\in \mathbf S_{t_0}$ with $v_i\in X_0(t_0)$, $i=1, 2$, we have 
\[
\|u_1(t)-u_2(t)\| \le C_\nu e^{-\nu (t-t_0)}\|v_1-v_2\|, \quad t\ge t_0.
\]
\end{theorem}
We now obtain the version of Theorem~\ref{HuyTM} for general dichotomies.
\begin{theorem}\label{thm:stable-manifold-mu}
Let $\mu \colon [0, \infty)\to [1, \infty)$ be a continuously differentiable growth rate with $\mu'(t)>0$ for each $t$. 
Suppose that $f$ satisfies the following:
\begin{enumerate}
\item[(a)] $f(t, 0)=0$ for $t\ge 0$;
\item[(b)] there is $c>0$ such that
\begin{equation}\label{lip2}
\|f(t, x)-f(t, y)\| \le c\frac{\mu'(t)}{\mu(t)}\|x-y\|, \quad \text{for $t\ge 0$ and $x, y\in X$.}
\end{equation}
\end{enumerate}
Moreover, suppose that $\{T(t, s)\}_{t\ge s\ge 0}$ admits a uniform $\mu$-dichotomy. Provided that $c$ is sufficiently small, there is an invariant stable manifold $\mathbf S$ for~\eqref{integeq}. Moreover, there are $\nu, C_\nu>0$ such that for each $t_0>0$ and solutions $u_1, u_2$ of~\eqref{integeq} on $[t_0, \infty)$ such that $u_i(t_0)=v_i+g_{t_0}(v_i)\in \mathbf S_{t_0}$ with $v_i\in X_0(t_0)$, $i=1, 2$, we have 
\[
\|u_1(t)-u_2(t)\| \le C_\nu \left (\frac{\mu(t)}{\mu(t_0)}\right )^{-\nu}\|v_1-v_2\|, \quad t\ge t_0.
\]
\end{theorem}

\begin{proof}
We define $h\colon \mathbb R^+\times X\to X$ by 
\[
h(t, x):=\frac{e^t}{\mu'(\mu^{-1}(e^t))}f(\mu^{-1}(e^t), x), \quad (t, x)\in \mathbb R^+\times X.
\]
Then $h(t, 0)=0$ for $t\ge 0$. Moreover, it follows from~\eqref{lip2} that 
\[
\|h(t, x)-h(t, y)\| \le c\|x-y\|, \quad \text{for $t\ge 0$ and $x, y\in X$.}
\]
In other words, $h$ satisfies the assumptions of Theorem~\ref{HuyTM}. Furthermore, since $\{T(t, s)\}_{t\ge s\ge 0}$ admits a uniform $\mu$-dichotomy, $\{T^{\mu, exp}(t, s)\}_{t\ge s \ge 0}$ admits an exponential dichotomy. We consider the associated integral equation
\begin{equation}\label{inteq2}
u(t)=T^{\mu, exp}(t, s)u(s)+\int_s^t T^{\mu, exp}(t, r)h(r, u(r))\, dr, \quad t\ge s.
\end{equation}
By Theorem~\ref{HuyTM}, \eqref{inteq2} has an invariant stable manifold $\mathbf{\bar S}$ such that for $t\ge 0$,
\[
\mathbf {\bar S}_t=\{w\in X: \ (t, w)\in \mathbf {\bar S}\}=\{v+g_t(v): \ v\in  X_0(t)\},
\]
where $g_t\colon  X_0(t)\to   X_1(t)$ and $X=X_0(t)\oplus  X_1(t)$, $t\ge 0$ is the splitting associated with $\mathbf {\bar S}$ (or dichotomy of $\{T^{\mu, exp}(t, s)\}_{t\ge s}$).

It remains to observe that if $t\mapsto u(t)$ is a (bounded)  solution of~\eqref{inteq2} on $[t_0, \infty)$, then  $t\mapsto u(\ln \mu(t))$ is a (bounded) solution of~\eqref{integeq} on $[\mu^{-1}(e^{t_0}), \infty)$, and if $t\mapsto u(t)$ is a  (bounded) solution of~\eqref{integeq} on $[t_0, \infty)$, then $t\mapsto u(\mu^{-1}(e^t))$ is a (bounded) solution of~\eqref{inteq2} on $[\ln \mu(t_0), \infty)$. The invariant stable manifold $\mathbf S$ for~\eqref{integeq} is thus given by 
\[
\mathbf S_t=\{v+g_{\ln \mu(t)}(v): \ v\in  X_0(\ln \mu(t))\}, \quad t\ge 0.
\]
\end{proof}

\begin{remark}
	Theorem~\ref{thm:stable-manifold-mu} is obtained from Huy's exponential	stable manifold theorem, Theorem 4.6 in~\cite{Huy}, through the	time rescaling $\tau=\ln\mu(t)$.
	In fact, a uniform $\mu$-dichotomy for the original evolution family becomes a uniform exponential dichotomy after this time change of variables. Moreover, the factor $\mu'(t)/\mu(t)$ in~\eqref{lip2} exactly compensates for the Jacobian of the change of time,	so that the transformed nonlinearity satisfies the uniform Lipschitz condition required in Theorem 4.6 in~\cite{Huy}. Finally, exponential contraction in the new time variable, $e^{-\nu(\tau-\tau_0)}$, becomes $\left(\mu(t)/\mu(t_0)\right)^{-\nu}$ in the original time variable. Thus, the invariant stable manifold, together with the  quantitative decay estimate for trajectories on it, is transported directly from Huy's
	exponential setting.
	
	Stable manifold theorems for nonautonomous equations under nonuniform exponential, nonuniform polynomial, and more general nonuniform dichotomies have also been established in~\cite{BV2,BS3,BS4}. Those results allow more general nonuniform linear dynamics and impose different conditions on the nonlinear perturbation. Therefore, they are not directly comparable with Theorem~\ref{thm:stable-manifold-mu}. 
	
	The specific feature of our result is that, under hypotheses exactly adapted to the time change	of variables $\tau=\ln\mu(t)$, a global Lipschitz stable manifold for a uniform $\mu$-dichotomy follows directly from the corresponding exponential result, without requiring a new graph-transform construction.
\end{remark}

	\section*{Acknowledgements}  
This paper has been funded by the European Union-NextGenerationEU-StatDin-MatAn-uniri-iz-25-108 (Davor Dragi\v cevi\' c).


\begin{thebibliography}{99}

\bibitem{B}
\newblock L. Backes,
\newblock A $(\mu, \nu)$-dichotomy spectrum,
\newblock \emph{J. Math. Anal. Appl.}  \textbf{549} (2025), Article: 129482.

\bibitem{BD}
\newblock L. Backes and D. Dragi\v cevi\' c,
\newblock A characterization of $(\mu, \nu)$-dichotomies via admissibility,
\newblock \emph{Math. Nachr.} \textbf{298} (2025), 2547--2569.

\bibitem{BDZ}
\newblock L. Backes, D. Dragi\v cevi\' c and W. Zhang,
\newblock Smooth linearization of nonautonomous dynamics under general dichotomic behaviour,
\newblock preprint, https://arxiv.org/pdf/2605.04184

\bibitem{BCV}
\newblock L. Barreira, J. Chu and C. Valls,
\newblock Lyapunov Functions for General Nonuniform Dichotomies,
\newblock \emph{Milan. J. Math.} \textbf{81} (2013), 153--169.

\bibitem{BDV}
\newblock L. Barreira, D. Dragi\v cevi\' c and C. Valls,
\newblock Admissibility on the half line for evolution families,
\newblock \emph{J. Anal. Math.} \textbf{132} (2017), 157--176.

\bibitem{BDV2}
\newblock L. Barreira, D. Dragi\v cevi\' c and C. Valls, 
\newblock Strong nonuniform spectrum for arbitrary growth rates, 
\newblock \emph{Commun. Contemp. Math.} \textbf{19} (2017),  1650008, 25 pp.

\bibitem{BV2}
\newblock  L. Barreira and C. Valls,
\newblock  Stable manifolds for nonautonomous equations without exponential dichotomy,
\newblock \emph{J. Differential Equations} \textbf{221} (2006), 58--90.

\bibitem{BV-DCDS}
\newblock L. Barreira and C. Valls, 
 \newblock Growth rates and nonuniform hyperbolicity, \newblock \emph{Discrete Contin. Dyn. Syst.} \textbf{22} (2008), 509--528.

\bibitem{BV-09}
\newblock L. Barreira and C. Valls, 
\newblock A Grobman-Hartman theorem for general nonuniform exponential dichotomies, 
\newblock \emph{J. Funct. Anal.} \textbf{257} (2009), 1976-1993.

\bibitem{BVJFA}
\newblock L. Barreira and C. Valls,
\newblock  Stable invariant manifolds for parabolic dynamics,
\newblock \emph{J. Funct. Anal.} \textbf{257} (2009), 1018--1029.

\bibitem{BV}
\newblock L. Barreira and C. Valls, 
\newblock Robustness of noninvertible dichotomies,
\newblock \emph{J. Math. Soc. Japan} \textbf{67} (2015), 293--317.

\bibitem{BS2} 
\newblock A. J. Bento and C. M. Silva, 
\newblock Nonuniform $(\mu,\nu)$-dichotomies and local dynamics of difference equations, 
\newblock \emph{Nonlinear Anal.} \textbf{75} (2012), 78--90.

\bibitem{BS3}
\newblock  A. J. G. Bento and C. M. Silva,
\newblock  Stable manifolds for non-autonomous equations with non-uniform polynomial
dichotomies,
\newblock \emph{Q. J. Math.} \textbf{63} (2012), 275--308.

\bibitem{BS4}
\newblock  A. J. G. Bento and C. M. Silva,
\newblock  Generalized nonuniform dichotomies and local stable manifolds,
\newblock  \emph{J. Dyn. Differential Equations} \textbf{25} (2013), 1139--1158.

\bibitem{CR}
\newblock A. Casta\~{n}eda and G. Robledo,
\newblock A topological equivalence result for a family of nonlinear difference systems having generalized exponential dichotomy, 
\newblock \emph{J. Difference Equ. Appl.} \textbf{22} (2016),  1271--1291.

\bibitem{CJ}
\newblock A. Casta\~{n}eda and N. Jara,
\newblock A generalization of Siegmund's normal forms theorem to systems with $\mu$-dichotomies,
\newblock \emph{J. Differential Equations} \textbf{410} (2024), 449--480.

\bibitem{CL}
\newblock C. Chicone and Y. Latushkin,
\newblock Evolution semigroups in dynamical systems and differential equations,
\newblock Math. Surveys Monogr., 70
\newblock American Mathematical Society, Providence, RI, 1999. x+361 pp.

\bibitem{Chu}
\newblock J. Chu,
\newblock Robustness of nonuniform behavior for discrete dynamics,
\newblock \emph{Bull. Sci. Math.} \textbf{137} (2013), 1031--1047.

\bibitem{CDS}
\newblock L. V. Coung, T. S. Doan and S. Siegmund, 
\newblock A Sternberg theorem for nonautonomous differential equations, 
\newblock \emph{J Dyn. Diff. Equat.} \textbf{31} (2019), 1279--1299.

\bibitem{D1}
\newblock D. Dragi\v cevi\'c, 
\newblock Admissibility and nonuniform polynomial dichotomies, 
\newblock \emph{Math. Nachr.} \textbf{293} (2019), 226--243.

\bibitem{D2}
\newblock D. Dragi\v cevi\' c, 
\newblock Admissibility and polynomial dichotomies for evolution families,
\newblock Commun. Pure Appl. Anal. \textbf{19} (2020),  1321--1336.

\bibitem{DD}
\newblock D. Dragi\v cevi\' c
\newblock $h$-dichotomies via noncritical uniformity and expansiveness for evolution families,
\newblock \emph{Math. Ann.} \textbf{395} (2026), 17pp.

\bibitem{DLP}
\newblock D. Dragi\v cevi\' c, N. Lupa and N. Jur\v cevi\' c Pe\v cek,
\newblock Admissibility and general dichotomies for evolution families,
\newblock \emph{Electron. J. Qual. Theory Differ. Equ.} (2020), Paper No. 58, 19 pp.

\bibitem{DSS1}
\newblock D. Dragi\v cevi\' c, A. L. Sasu and B. Sasu, 
\newblock On Polynomial Dichotomies of Discrete Nonautonomous
Systems on the Half-Line, 
 \newblock \emph{Carpathian J. Math.} \textbf{38} (2022), 663--680.

\bibitem{DS}
\newblock D. Dragi\v cevi\' c and C. M. Silva,
\newblock Generalized dichotomies via time rescaling,
\newblock \emph{Discrete and Continuous Dynamical Systems S} \textbf{18} (2025), 3917-3944.

\bibitem{Huy}
\newblock N. T. Huy,
\newblock Invariant manifolds of admissible classes for semi-linear evolution equations, 
\newblock \emph{J. Differential Equations} \textbf{246} (2009), 1820--1844.

\bibitem{GJ}
\newblock C. A. Gallegos and N. Jara,
\newblock The interplay of $\mu$-dichotomy, bounded growth, and spectral properties via growth rate comparisons,
\newblock J. Differential Equations \textbf{480} (2026), Paper No. 114645, 24pp.


\bibitem{JG}
\newblock N. Jara and C. A. Gallegos,
\newblock Spectrum invariance dilemma for nonuniformly kinematically similar systems,
\newblock \emph{Math. Ann.} \textbf{391} (2025), 2255--2280.

\bibitem{JLR}
\newblock N. Jara, I. P. Longo and M. Rasmussen,
\newblock Dichotomies Faster or Slower than exponential are Irrelevant for Skew-Product Flows,
\newblock preprint, https://arxiv.org/pdf/2605.23173

\bibitem{Martin}
\newblock R. H. Martin, Jr., 
\newblock Conditional stability and separation of solutions to differential equations, 
\newblock J. Differential Equations \textbf{13} (1973), 81--105.

\bibitem{M}
\newblock J. S. Muldowney, 
\newblock Dichotomies and asymptotic behaviour for linear differential systems, 
\newblock \emph{Trans. Amer. Math. Soc.} \textbf{283} (1984), 465--484.

\bibitem{NP}
\newblock R. Naulin and M. Pinto, 
\newblock Roughness of $(h,k)$-dichotomies, 
\newblock \emph{J. Differential Equations} \textbf{118} (1995), 20--35.

\bibitem{Per}
\newblock O. Perron, 
\newblock Die Stabilit\"atsfrage bei Differentialgleichungen, 
\newblock \emph{Math. Z.} \textbf{32} (1930), 703--728.

\bibitem{Siegmund}
\newblock S. Siegmund
\newblock Dichotomy spectrum for nonautonomous differential equations
\newblock \emph{J. Dyn. Diff. Equat.} \textbf{14} (2002), 243--258.

\bibitem{SilvaCPAA}
\newblock C. Silva, 
\newblock Admissibility and generalized nonuniform dichotomies for discrete dynamics, 
\newblock \emph{Comm. Pure Appl. Anal.} \textbf{20} (2021), 3419--3443.

\bibitem{Silva}
\newblock C. Silva,
\newblock Nonuniform $\mu$-dichotomy spectrum and kinematic similarity, 
\newblock \emph{J. Differential Equations}  \textbf{375} (2023), 618--652.

\bibitem{Torres}
\newblock F. Torres,
\newblock A Green function approach to smooth nonautonomous topological equivalence with unbounded nonlinearities under $(\mu, \nu)$-dichotomies,
\newblock preprint, https://arxiv.org/pdf/2608.14715

\end{thebibliography}
\end{document}